\documentclass[11pt]{article}

\def\titlename{\huge Multipliers of Beurling-Fourier algebras}

\title{\titlename}

\def\authname{Mahmood Alaghmandan, Olof Giselsson, Ebrahim Samei, Lyudmila Turowska}
\author{{\normalsize\sc \authname}}

\usepackage{amsmath}

\usepackage{yfonts}

\usepackage{xcolor}
\definecolor{blue1}{RGB}{32,78,170}
\definecolor{blue2}{RGB}{93,92,160}
\definecolor{blue3}{RGB}{40,51,202}
\definecolor{blue4}{RGB}{0,0,0}
\definecolor{purple1}{RGB}{128,0,128}

\usepackage{geometry}
\usepackage{marginnote}
\newcommand{\mn}[1]{{\marginpar{ \begin{center}\footnotesize\color{blue}{\textsf{#1}}\end{center}}}}

\usepackage[all]{xy}

\definecolor{El}{rgb}{.4,.9,1}

\usepackage[T1]{fontenc}
\usepackage{marvosym}

\usepackage{titlesec}

\titleformat{\section}
  {\normalfont\fontsize{12}{13}\bfseries}{\thesection}{1em}{}

\usepackage[english]{babel}
\usepackage{blindtext}
\usepackage[scaled=.90]{helvet}
\usepackage{xcolor}
\usepackage{titlesec}
\titleformat{\chapter}[display]
  {\normalfont\sffamily\huge\bfseries\color{blue4}}
  {\chaptertitlename\ \thechapter}{20pt}{\Huge}
\titleformat{\section}
  {\normalfont\sffamily \large \color{blue4}}
{\thesection}{1em}{}

 \titleformat{\subsection}
  {\normalfont\sffamily\large\color{blue4}}
  {\thesubsection}{1em}{}

            \newcounter{pulse}[section]

\numberwithin{pulse}{section}
\numberwithin{equation}{section}

\newtheorem{theorem}[pulse]{\bf \textsf{Theorem}}
\newtheorem{proposition}[pulse]{\bf \textsf{Proposition}}
\newtheorem{lemma}[pulse]{\bf \textsf{Lemma}}
\newtheorem{corollary}[pulse]{\bf \textsf{Corollary}}

\newtheorem{lem}[pulse]{\bf \textsf{Lemma}}
\newtheorem{cor}[pulse]{\bf \textsf{Corollary}}

\newtheorem{dummy-eg}[pulse]{\bf \textsf{Example}}
\newtheorem{dummy-rem}[pulse]{\bf \textsf{Remark}}
\newenvironment{eg}{\begin{dummy-eg}\upshape}{\end{dummy-eg}\ignorespacesafterend}
\newenvironment{rem}{\begin{dummy-rem}\upshape}{\end{dummy-rem}\ignorespacesafterend}

\newtheorem{dummy-def}[pulse]{\bf \textsf{Definition}}

\usepackage{amsmath,amssymb}

\newenvironment{dfn}{\begin{dummy-def}\upshape}{\end{dummy-def}\ignorespacesafterend}

\newenvironment{proof}{\noindent{\it Proof.}\/}{\hfill$\Box$\newline\ignorespacesafterend}

\newcommand{\cA}{{\cal A}}
\newcommand{\cB}{{\cal B}}
\renewcommand{\ker}{\operatorname{Ker}}

\newcommand{\cK}{\mathcal{K}}

\newcommand{\om}{\omega} 

\newcommand{\norm}[1]{\Vert #1 \Vert}

\newcommand{\Ind}{{\bf I}}

\newcommand{\bG}{{\Bbb G}}

\newcommand{\cH}{{\mathcal{H}}}
\newcommand{\id}{\operatorname{id}}
\newcommand{\bm}{{\bf m}}
\newcommand{\tbm}{{\tilde{\bf m}}}
\newcommand{\bn}{{\bf n}}
\newcommand{\tbn}{{\tilde{\bf n}}}
\newcommand{\rang}{\operatorname{Ran}}

\begin{document}
\date{}
\maketitle
\begin{abstract}
For a locally compact group 
$G$
 we introduce and study the reduced Beurling–Four\-ier–Sti\-eltjes algebra, a weighted analogue of the reduced Fourier–Stieltjes algebra 
$B_r(G)$, together with the algebra of completely bounded multipliers of the associated weighted Fourier algebra. We show, in particular, that these two algebras coincide when $G$
is amenable. For a general locally compact group 
$G$,
 we identify them as subspaces of $B_r(G)$
 and of the space of functions that locally belong to the Fourier algebra, respectively. Furthermore, we establish sufficient conditions on the group and the weight under which the algebra of completely bounded multipliers of the weighted Fourier algebra embeds into its unweighted counterpart.

\end{abstract}


\vskip2.5em
\tableofcontents

\section{Introduction}
One of the standard function spaces used in the analysis of a locally compact group 
$G$ is the Fourier algebra 
$A(G)$, introduced by P. Eymard \cite{ey}. It is defined as the unique predual of the group von Neumann algebra 
$VN(G)$, that is, as the associated noncommutative 
$L_1$-space. It is well known that 
$A(G)$ can be realized as a Banach subalgebra of 
$C_0(G)$, and moreover, it determines the underlying group.

While $A(G)$ has the drawback of a complicated norm structure, it has the significant advantage of being commutative. This allows one to study its Gelfand spectrum $\text{spec}A(G)$, which, as shown by Eymard, is canonically homeomorphic to 
$G$ itself. This provides a nontrivial link between the structures of topological groups and Banach algebras.

On the other hand, the spaces of multipliers and completely bounded multipliers of 
$A(G)$ have played a central role in the study of group properties such as amenability, weak amenability, the Haagerup property, and various approximation properties for the associated group 
$C^*$ - and von Neumann algebras (see e.g.  \cite{brown-ozawa,canniere-haagerup, cowling-haagerup}). For example, a locally compact group 
$G$ is amenable if and only if the multiplier algebra of 
$A(G)$ coincides with the Fourier–Stieltjes algebra 
$B(G)$, the algebra generated by positive definite functions on 
$G$. The study of multiplier spaces has thus been a fundamental topic in abstract harmonic analysis \cite{boz, canniere-haagerup, cowling-haagerup,brown-ozawa}.

In recent years, several authors, including the second, third, and fourth authors, have investigated weighted versions of the Fourier algebra 
$A(G)$, known as Beurling–Fourier algebras 
$A(G,\omega)$, obtained by introducing weights that modify the norm structure of 
$A(G)$ \cite{LeSa, lst, ors, gllst, gh-lee, gt,lee_lee}. When 
$G$ is abelian, such algebras are isomorphic—via the Fourier transform—to classical Beurling algebras on the dual group $\hat G$, which explains the terminology.

In \cite{lst,gllst,gt, lee_lee}, the Gelfand spectrum of Beurling–Fourier algebras was studied and realized as a subset of a suitable complexification of the underlying group, revealing interesting connections between harmonic analysis and geometry. Furthermore, in \cite{ors}, the existence of a bounded approximate identity for $A(G,\omega)$
 was characterized in terms of amenability of $G$, extending the classical Leptin–Herz theorem. More recently, weighted Fourier algebras have also been studied in the setting of compact quantum groups \cite{fl, lee_voigt}.

In this paper, we initiate the study of a weighted analogue $B_r(G,\omega)$ of the reduced Fourier-Stieltjes algebra $B_r(G)$ defined using a weight inverse $\omega$ on the dual of $G$.
 We realize $B_r(G,\omega)$
 as a subspace of $B_r(G)$.
 Since the dual of a Beurling–Fourier algebra $A(G,\omega)$
 is still $VN(G)$,
 one can consider its completely bounded multipliers $M_{\rm cb}A(G,\omega)$.
 Viewing these as functions on $G$
 we prove that
 $$B_r(G,\omega)\subset M_{\rm cb}A(G,\omega)$$
and show that this inclusion is proper if and only if $G$
is non-amenable, extending the corresponding result in the unweighted case. 
Building on idea's of Jolissaint in \cite{jol} when proving a characterisation of completely bounded multipliers of $A(G)$, as Herz-Schur multipliers, we show that $M_{\rm cb} A(G,\omega)$ can be realized as  a subspace of $A(G)^{\rm loc}$, the space of functions which locally belongs to $A(G)$. We further identify a broad class of groups for which $M_{\rm cb}A(G,\omega)$ is a subset of the completely bounded multipliers for the unweighted case. This is achieved by relating the problem to the structure of closed convex hulls of $G$-orbits of the multiplier algebra of the reduced $C^*$-algebra $C_r^*(G)$ of $G$ as well as the existence of traces on $C_r^*(G)$. Finally, we establish this inclusion for arbitrary groups under natural restrictions on the weight, such as central weights and weights induced from normal amenable subgroups, which constitute one of the main sources of examples of weights.

\section{Preliminaries}\label{s:perliminaries}

Let $A$ be a C$^*$-algebra, and let $(\pi , \cH)$ be a $*$-representation of $A$. The closed linear span of $\{\pi(x)\xi : x \in A, \xi \in \cH\}$ is called the \emph{essential subspace} of $(\pi, \cH)$. The $*$-representation $\pi$ of $A$ is said to be \emph{nondegenerate}, if its essential subspace is $\cH$.

Let  $A^{**}$ be the second conjugate of $A$ and see it as the enveloping von Neumann algebra. We define the \emph{multiplier algebra} of $A$ as 
 \[
M(A)=\{x \in A^{**}: xA \cup A x \subseteq A\}.
\]
It is easy to check that $M(A)$ is a closed $C^*$-subalgebra of $A^{**}$ and $M(A)=A$ if and only if $A$ is unital; $A$ is moreover a closed ideal in $M(A)$. 
If $(\pi, \cH)$ is a faithful nondegenrate representation of $A$, then $M(A)$ is $*$-isomorphic to
\[
\{ x \in B(\cH):   x \pi(A) \cup \pi(A) x \subseteq \pi(A)\},
\]
see \cite[Proposition~2.12.9]{ren}.

Consider two $C^*$-algebras $A$ and $B$ and a $*$-homomorphism $\pi : A\rightarrow M (B)$. It is called \emph{non-degenerate} if $\pi(A)B := \{ \pi(a)b : a\in A, b \in B \}$ is dense in $B$.


Let $G$ be a locally compact group, and let  $\lambda^G$ be the left regular representations of $G$. We will write $\lambda$ instead of $\lambda^G$ when there is no fear of ambiguity. We write $VN(G)$ and $C_r^*(G)$ for the von Neumann algebra and the reduced $C^*$-algebra of $G$, respectively, associated to $\lambda$. Let
\[
\Gamma: VN(G) \rightarrow VN(G) \bar{\otimes} VN(G))
\]
be the \emph{co-multiplication} on $VN(G)$ defined by $\Gamma(\lambda_t) = \lambda_t \otimes \lambda_t$ for every $t \in G$. Its restriction to $C^*_r(G)$, which we still denote by $\Gamma$, is a $*$-homomorphism
\[
\Gamma: C^*_r(G) \rightarrow M(C^*_r(G) \otimes C^*_r(G)).
\]
Here and below we write $A \otimes B$ for the minimal tensor product of $C^*$-algebras $A$ and $B$.
Therefore,
\[
\Gamma^*: M(C^*_r(G) \otimes C^*_r(G))^* \rightarrow C^*_r(G)^*=B_r(G),
\]
where $B_r(G)$ is the reduced Fourier-Stieltjes algebra of $G$. 
As $\Gamma$ is a non-degenerate $*$-homomorphism, it extends uniquely to a $*$-homomorphism
\[
\Gamma: M(C^*_r(G)) \rightarrow M(C^*_r(G) \otimes C^*_r(G)).
\]
We point out the following fact that we will use repeatedly throughout the article: The co-multiplication $\Gamma$ has the property that either of the the sets
$\{(1\otimes y)\Gamma(x): x,y \in C^*_r(G)\}$ and $\{(y \otimes 1)\Gamma(x): x,y \in C^*_r(G)\}$ spans a dense subset of $C^*_r(G\times G)\simeq C_r^*(G)\otimes C_r^*(G)$.

The topological dual $B_r(G)=C^*_r(G)^*$ is a Banach algebra under the multiplication $B_r(G) \times B_r(G) \rightarrow B_r(G)$ given by
\begin{equation}\label{eq:convolution}
(uv)(x)=\Gamma(x)(u\otimes v)
\end{equation}
for all $u,v \in B_r(G)$ and $x \in C^*_r(G)$.  We note that, in  \eqref{eq:convolution}, we consider $\Gamma(x)$ as an element in $(C^*_r(G) \otimes C_r^*(G))^{**}\supset M(C^*_r(G) \otimes C_r^*(G)) $. 

Finally, $B_r(G)$ can be  also identified with the weak* closure in the Fourier-Stieltjes algebra $B(G)\subset C_b(G)$ of the matrix coefficients of the left regular representation with multiplication being pointwise multiplication. More precisely, $B_r(G)$ is the algebra of matrix coefficients of representations $\pi$ of $G$ weakly associated to $\lambda$.  For basic properties of $B_r(G)$ and $B(G)$ we refer the reader to \cite{kaniuth-lau}.

\section{Weight inverses on the dual of locally compact groups}

There are different ways to define weights on the dual of locally compact groups, see \cite{gllst, LeSa, lst}. In this paper, we adopt the one proposed in \cite{ors}.

\begin{dfn}\label{D:Weight}
(\cite[Definition 2.1]{ors})
Let $G$ be a locally compact group $G$. A \emph{weight inverse} is an element $\omega $  of $M(C^*_r(G))$ with $\|\omega \|\leq 1$ such that the following are satisfied:
\begin{itemize}
\item[(1)]{The maps
\[
C^*_r(G) \rightarrow C^*_r(G), x \mapsto x\omega 
\]
and
\[
C^*_r(G) \rightarrow C^*_r(G), x \mapsto \omega x
\]
have dense ranges;}
\item[(2)]{there is $\Omega  \in VN(G\times G)$ with $\|\Omega  \| \leq 1$ such that
\[
\omega  \otimes \omega = \Gamma(\omega ) \Omega.
\]}
\end{itemize}
\end{dfn}

Below, we highlight some properties of inverse weights that we will use further in the article.

\begin{lem}\label{mult}
The element $\Omega \in VN(G\times G)$ is a contractive right multiplier of $C^*_r(G)\otimes C_r^*(G)$, such that $\Gamma(x)\Omega\in M(C^*_r(G)\otimes C_r^*(G))$ for all $x\in C^*_r(G).$ Moreover, for any $u,v\in B_r(G)$, the element $\Omega (u\otimes v)\in (C_r^*(G)\otimes C_r^*(G))^*$, 
determined by 
\begin{equation}\label{omuv}
\begin{array}{ccc}
(\Omega (u\otimes v))(T):=(u\otimes v)(T\Omega ), & \text{for $T\in C_r^*(G)\otimes C_r^*(G)$,}
\end{array}
\end{equation}
extends to a bounded linear functional on $ M(C_r^*(G)\otimes C_r^*(G))$.
\end{lem}

\begin{proof}
For any  $x,y\in C^*_r(G)$, we have
$$
(1 \otimes y)\Gamma(x\omega )\Omega =(1 \otimes y)\Gamma(x)(\omega \otimes \omega )\in C_r^*(G)\otimes C_r^*(G).$$
Thus, by condition $(1)$ on $\omega$ (from Definition \ref{D:Weight}) and the fact that the set $\{(1 \otimes y)\Gamma(z): y,z \in C^*_r(G)\}$ spans a dense subset of $C_r^*(G)\otimes C_r^*(G)$, we have 
$$(C^*_r(G)\otimes C_r^*(G))\Omega \subseteq C^*_r(G)\otimes C_r^*(G).$$ 
Furthermore, this implies that $\Omega (u\otimes v)$ as defined by~\eqref{omuv} is a linear bounded functional on $C_r^*(G)\otimes C_r^*(G)$.
Thus it can be extended to a bounded linear functional on $M(C_r^*(G)\otimes C_r^*(G))$ as any bounded linear function on a C*-algebra extends (uniquely) to a bounded linear functional on its multiplier algebra. 
\end{proof}

\begin{lem}\label{l:ker(w-1)}
Let $G$ be a locally compact group with a weight inverse $\omega $. Then the kernels of $\omega $ and $\omega^*$ as  operators in $\cB(L^2(G))$, and kernels of $\Gamma(\omega )$ and $\Gamma(\omega )^*$ as operators in $\cB(L^2(G\times G))$ are trivial.   Subsequently, $\ker (\Omega )=\{0\}$.
\end{lem}

\begin{proof}
Note that the set $\{x   \xi: x\in C^*_r(G), \xi \in L^2(G)\}$ is dense in $L^2(G)$.  Also, $\omega  C^*_r(G)$ is dense in $C^*_r(G)$ by condition $(1)$ of Definition \ref{D:Weight}. Subsequently, we get $\{ \omega  x   \xi: x\in C^*_r(G), \xi \in L^2(G)\}$ is dense in $L^2(G)$. Therefore, the closure of $\rang(\omega )$ is equal to $L^2(G)$ and hence $\ker \omega^*$ is trivial, as desired.

To prove the claim for $\Gamma(\omega )^*$, note that the linear span of $\{\Gamma(x)(1\otimes y): x,y \in C^*_r(G)\}$ is dense in $C^*_r(G \times G)$. Therefore, following the fact that $\omega C^*_r(G)$ is dense in $C^*_r(G)$, we have that  the linear span of
\[
\{\Gamma(\omega x)(1\otimes y) \eta: x,y\in C^*_r(G), \eta \in L^2(G\times G)\}
\]
is dense in $L^2(G\times G)$. Subsequently, $\Gamma(\omega )$ has a dense range.

Similarly, as $\{(x\omega )^*: x\in C_r^*(G)\}$  is dense in $C_r^*(G)$ we get that the range of $\omega^*$ is dense in $L^2(G)$ and subsequently $\ker\omega =\{0\}$. The arguments for $\ker\Gamma(\omega )=\{0\}$ are similar.
Now if $\Omega $ has a non-zero kernel, then $\Gamma(\omega )\Omega =\omega  \otimes \omega $ cannot have a trivial kernel, a contradiction.
\end{proof}
\begin{rem}
    In \cite{gt} a weight inverse was defined equivalently as $\omega\in VN(G)$ which satisfies condition (2) of Definition \ref{D:Weight} and the condition $\ker\omega=\ker\omega^*=\{0\}$. Hence, by  Lemma \ref{l:ker(w-1)}, a weight inverse from Definition \ref{D:Weight}  is a weight inverse in the sense of \cite{gt}. This will be used when referring to \cite{gt} for properties of $\omega$ and related algebras. 
\end{rem}
\begin{proposition}\label{p:weights-product}
Let $G_1$ and $G_2$ be locally compact groups and $\omega _1$ and $\omega _2$ be two weight inverses on $G_1$ and $G_2$ respectively. Then $\omega _1 \otimes \omega _2$ is a weight inverse on $G_1 \times G_2$.
\end{proposition}

\begin{proof}
Let $\Gamma^k$ be the co-multiplication of $VN(G_k)$ for $k=1,2$. Then for co-multiplication of $VN(G_1 \times G_2)$, denoted by $\Gamma^{12}$, we have
\[
\Gamma^{12} (x \otimes y) = (\id \otimes \Sigma \otimes \id)( \Gamma^1(x) \otimes \Gamma^2(y))
\]
where $x \in VN(G_1)$ and $y \in VN(G_2)$ and $\Sigma$ is the flip map
\[
\Sigma: VN(G_2) \bar{\otimes} VN(G_1) \rightarrow VN(G_1) \bar{\otimes} VN(G_2).
\]
Hence
\begin{eqnarray*}
( \omega_1 \otimes \omega_2) \otimes (\omega_1 \otimes \omega_2)
&=& \id \otimes \Sigma \otimes \id \left( \omega_1 \otimes \omega_1   \otimes \omega_2 \otimes \omega_2 \right)\\
&=& \id \otimes \Sigma \otimes \id \left( \Gamma^1(\omega_1) \otimes \Gamma^2(\omega_2 )\right)  \id \otimes \Sigma \otimes \id (\Omega _1 \otimes \Omega _2) \\
&=& \Gamma^{12} \left( \omega_1 \otimes \omega_2 \right)    \id \otimes \Sigma \otimes \id (\Omega_1 \otimes \Omega_2).
\end{eqnarray*}
Finally, we note that the maps $x \mapsto (\omega_1 \otimes \omega_2) x$ and $x \mapsto x (\omega_1 \otimes \omega_2)$, $x \in C^*_r(G_1\times G_2)$, have dense ranges as $C^*_r(G_1) \odot C^*_r(G_2)$ is dense in $C^*_r(G_1\times G_2)$. Hence,  both conditions of Definition \ref{D:Weight} satisfy for $\omega_1\otimes \omega_2$ so it is a weight inverse on $G_1\times G_2$.
\end{proof}


\section{Beurling-Fourier algebras}\label{s:Beurling-Fourier-Stieltjes}

We can use weight inverse to construct the ``non-commutative" analogue of Beurling algebras on locally compact groups, see for example \cite{gllst, LeSa, lst, ors}.

Recall first the Fourier algebra $A(G)$ of $G$. It is   the unique pre-dual of $VN(G)$ and can be identified with the space of functions on $G$:  $$A(G)=\{g\ast\check h\;|\; g,h\in L^2(G)\}\subset C_0(G),$$ where $\check h(s)=h(s^{-1})$, $s\in G$, and 
$$\xi\ast\check \eta(s)=\int \xi(t)\eta(s^{-1}t)dt=\langle\lambda_s\eta,\bar\xi\rangle;$$
$A(G)$ becomes a 
 commutative Banach algebra with respect to the pointwise multiplication and the norm given by
$$\|f\|_{A(G)}=\inf\|\xi\|_2\|\eta\|_2,$$ where the infimum is taken over all possible decomposition $f=\xi\ast\check \eta$, see for example  \cite{kaniuth-lau}. The duality between $VN(G)$ and $A(G)$ is given by
$$\langle u,T\rangle=(T \xi,\eta)$$ for $T\in VN(G)$ and $u(s)=(\lambda_s\xi,\eta)=(\bar \eta\ast\check \xi)(s)\in A(G);$ here and through the rest of the paper we use $(\cdot,\cdot)$ to denote the inner product on a Hilbert space and we keep notation $\langle\cdot,\cdot\rangle$ for duality pairing between $\mathcal M$ and $\mathcal M^*$. 

For $T\in VN(G)$ and $f\in A(G),$ we let $Tf\in A(G)$ be given by
$$
\begin{array}{ccc}
\langle Tf, R\rangle:=\langle f,RT\rangle, &\text{for $R\in VN(G).$}
\end{array}
$$
The assignment $(T,f)\mapsto Tf$ turns $A(G)$ into a left $VN(G)$-module.

\begin{dfn}\label{d:Beurling-Fourier} 
Let $G$ be a locally compact group and $\omega$ a weight inverse on it.
The corresponding \emph{Beurling-Fourier algebra} is defined as 
$$A(G, \omega) := \omega A(G)=\{\omega u : u \in A(G)\},$$ with 
$$\norm{\omega u}_{A(G, \omega)}:=\norm{u}_{A(G)} \ \ \ (\forall u \in A(G)).$$
\end{dfn}
We have, for each pair $\omega u, \omega v \in A(G, \omega)$,
$$
(\omega u)( \omega v)= \omega  \Gamma_*(\Omega (u \otimes v))
$$
and $\|\cdot\|_{A(G,\omega)}$ is a norm making $A(G,\omega)$  a commutative Banach algebra. Moreover, the dual of $A(G,\omega)$ is $VN(G)$, 
where the duality is given by 
\[
\langle \omega u, x\rangle_\omega  := \langle u, x\rangle, x \in VN(G), u \in A(G),
\]
(see e.g. \cite[Proposition 2.3]{gt}).
In particular, for each $u \in A(G)$ and $s \in G$ we get
\begin{equation}\label{eq:point-evaluation}
(\omega u)(s) = \langle \omega u, \lambda_s \rangle = \langle \omega u, \lambda_s \omega \rangle_\omega .
\end{equation}
It follows from \cite{ors} that $A(G,\omega)$ is a completely contractive commutative Banach algebra (see also
\cite{ gllst,gt, LeSa, lst}). Furthermore, the spectrum of $A(G,\omega)$ has been studied in \cite{gllst, gt, lst}.
The following has been proved in \cite{gllst}. 
\begin{theorem}\label{l:G->D(A(G,w))}
Let $G$ be a locally compact group with a weight inverse $\omega $. Then  there is a continuous injective embedding $\iota$ from topological space $G$  into the Gelfand spectrum of $A(G,\omega )$.
\end{theorem}




The following statement shows that the completely contractive Banach algebra structure on $A(G,\omega)$ also behaves well with respect to the operator space projective tensor product.

\begin{proposition}\label{p:A(w1xw2)}
Let $G_1$ and $G_2$ be locally compact groups and $\omega _1$ and $\omega _2$ be two weight inverses on $G_1$ and $G_2$ respectively.  Then $A(G_1 \times G_2, \omega _1 \otimes \omega _2)$ is completely isometrically isomorphic to $A(G_1, \omega _1) \hat{\otimes} A(G_2,\omega _2)$ as completely contractive Banach algebras.
\end{proposition}

\begin{proof}
Note that both of the operator spaces $A(G_1 \times G_2, \omega _1 \otimes \omega _2)$ and $A(G_1, \omega _1) \hat{\otimes} A(G_2,\omega _2)$ are preduals of $VN(G_1)\bar{\otimes}VN(G_2)$. Therefore, they are completely isometrically isomorphic through the mapping
\[
\iota: (\omega _1 \otimes \omega _2)(u_1\otimes u_2) \mapsto (\omega _1u_1) \otimes (\omega _2u_2)
\]
where $u_k \in A(G_k)$ for $k=1,2$.
So we only need to prove that this mapping is an algebra homomorphism. Note that for each pair $u_k, v_k \in A(G_k)$, $k=1,2$,
\[
\left((\omega _1 u_1) \otimes (\omega _2 u_2), (\omega _1 v_1) \otimes (\omega _2  v_2)\right) \mapsto
\left( \omega _1  \Gamma^1_*(\Omega _1 (u_1 \otimes v_1))\right) \otimes \left(\omega_2  \Gamma^2_*(\Omega_2 (u_2 \otimes v_2)) \right)
\]
is the multiplication in $A(G_1, \omega _1) \hat{\otimes} A(G_2,\omega _2)$. Therefore, we get

\begin{eqnarray*}
&& \iota( (\omega _1  \otimes \omega _2 ) \Gamma^{12}_* \left( (\id \otimes\Sigma \otimes\id) (\Omega _1 \otimes\Omega _2) (u_1 \otimes u_2) \otimes (v_1 \otimes v_2) \right)\\
  &&= \iota( (\omega _1  \otimes \omega _2 ) (\id \otimes \Sigma \otimes \id)( \Gamma^1_*  \otimes \Gamma^2_* ) \left( (\id \otimes\Sigma \otimes\id) (\Omega _1 \otimes\Omega _2) (u_1 \otimes u_2) \otimes (v_1 \otimes v_2) \right)\\
  &&= \iota( (\omega _1  \otimes \omega _2 ) (  \Gamma^1_*  \otimes \Gamma^2_* )\left( (\Omega _1 \otimes\Omega _2) (u_1 \otimes v_1) \otimes (u_2  \otimes v_2) \right)\\
    &&=  \omega _1  \Gamma^1_*( \Omega _1 (u_1 \otimes v_1))) \otimes \omega _2 (   \Gamma^2_*  \left(  \Omega _2  (u_2  \otimes v_2) \right).
\end{eqnarray*}

\end{proof}

\section{Reduced Beurling-Fourier-Stieltjes spaces}

In this section we define a weighted analog of the reduced Fourier-Stieltjes algebra. 

For $u\in B_r(G)$ and $x\in M(C_r^*(G))$, we let $xu$ and $ux$ in $B_r(G)$ be given by
$$\langle xu, y\rangle:=\langle u, yx\rangle \ \ \text{and}\ \ \langle ux, y\rangle:=\langle u, xy\rangle \ \ \ (y\in C_r^*(G)).$$

\begin{dfn}\label{d:Beurling-FS}
The reduced Beurling-Fourier-Stieltjes  space of $G$ is defined to be the Banach space
\[
B_r(G,\omega ):=\{ \omega u: u \in B_r(G)\}
\]
with $\norm{\omega u}_{B_r(G, \omega )}:=\norm{u}_{B_r(G)}$ for each $u \in B_r(G)$.
\end{dfn}
    If $u\in B_r(G)$, then $u(t)=(\pi(t)\xi,\eta)$, $t\in G$, for a non-degenerate representation $\pi$ of $G$ weakly associated to $\lambda$, and then  $(\omega u)(t)=(\pi(t)\pi(\omega)\xi,\eta)$, $t\in G$.  
We note also that the norm is well-defined follows from the density of the map $C_r^*(G)\to C_r^*(G)$, $x\mapsto x\omega $. Also, 
we have
$A(G, \omega )$ is a subset of $A(G)$ and $B_r(G, \omega )$ is a subset of $B_r(G)$. The following shows that equality for the corresponding algebras occurs only in the trivial cases. 

\begin{cor}\label{c:omega-invertible}
Let $G$ be a locally compact group with a weight inverse $\omega $. Then  the following are equivalent.
\begin{itemize}
\item[(i)]{The inclusion $A(G, \omega )$ into $A(G)$ is surjective.}
\item[(ii)]{The inclusion $B_r(G, \omega )$ into $B(G)$ is surjective.}
\item[(iii)]{ $\omega $ is invertible in $M(C^*_r(G))$.}
\item[(iv)]{$\omega $ is invertible in $VN(G))$.}
\end{itemize}
\end{cor}

\begin{proof}
(i) $\Leftrightarrow$(iii)$\Leftrightarrow$(iv) follows from \cite[Proposition 2.6]{gt}. (iii)$\Rightarrow$(ii) is trivial.  

(ii)$\Rightarrow$(iii). Suppose that the map $\bm_{\omega }: B_r(G) \rightarrow B_r(G)$, $u \mapsto \omega u$, is a bijection, that is $\bm_{\omega}$ is invertible. Then so is
\[
\bm_{\omega }^*: C^*_r(G)^{**} \rightarrow C^*_r(G)^{**}
\]
which is indeed given by $\bm_{\omega }^*(x) = x \omega $ ($\forall x \in C^*_r(G)^{**}$). 
But since $C^*_r(G)^{**}$ has an identity, there is some $\omega^{-1}  \in C^*_r(G)^{**}$ so that $\bm_{\omega }^*(\omega^{-1} ) = 1$ i.e. $\omega^{-1}  \omega  = 1$. This implies that $\omega\omega^{-1}(\omega x)=\omega x$ for any $x\in C_r^*(G)$. However, $\omega (C_r^*(G))$ is dense in $C_r^*(G)$ and $C_r^*(G)$ is weak*-dense in $C_r^*(G)^{**}$, showing that $\omega\omega^{-1}=1$. As $M(C^*_r(G))$ is a  C*-subalgebra of $C^*_r(G)^{**}$ and  C$^*$-algebras are inverse closed, $\omega^{-1}\in M(C_r^*(G))$.



\end{proof}

 We call a Banach algebra  $\cA$ \emph{faithful} if it satisfies the following property: for $b\in \cA$, if $aba' = 0$ for all $a, a' \in \cA$, then $b = 0$.

\begin{theorem}\label{t:main}
Let $G$ be a locally compact group and $\omega $ be a weight inverse on it. Then $B_r(G, \omega )$ is a completely contractive faithful Banach algebra with pointwise multiplication. In addition, $A(G, \omega )$ is an ideal in $B_r(G, \omega )$.
\end{theorem}

\begin{proof}
Let $u,v \in B_r(G)$. By Lemma~\ref{mult}, we have $\Omega  (u \otimes v) \in M(C^*_r(G) \otimes C^*_r(G))^*$. Therefore, $ \Gamma^*(\Omega  (u \otimes v)) \in C^*_r(G)^*=B_r(G)$.  Consequently, $\omega  \Gamma^*(\Omega  (u \otimes v)) \in B_r(G,\omega )\subset B_r(G)$.
Then, for each $x \in C^*_r(G)$, we have
\begin{eqnarray*}
\langle    \omega  \Gamma^*(\Omega  (u \otimes v)), x\rangle &=&  \langle   \Gamma^*(\Omega  (u \otimes v)), x \omega  \rangle \\
&=& \langle   \Omega  (u \otimes v), \Gamma(x \omega ) \rangle \\
&=& \langle   \Omega  (u \otimes v),\Gamma( x)\Gamma(\omega ) \rangle \\
&=&  \langle \Gamma(\omega ) \Omega  (u \otimes v), \Gamma(x)\rangle \\
&=& \langle (\omega  \otimes \omega ) (u \otimes v), \Gamma(x)\rangle \\
&=& \langle (\omega u) (\omega v), x\rangle,
\end{eqnarray*}
where the first two pairings and the last one is the $(B_r(G), C_r^*(G))$-pairing and the rest is the one between $M(C_r^*(G)\otimes C_r^*(G))^*$ and $M(C_r^*(G)\otimes C_r^*(G))$.

Hence $(\omega u) (\omega v)=\omega  \Gamma^*(\Omega  (u \otimes v))$ and therefore it  belongs to $B_r(G, \omega )$. Furthermore, since the mappings $\Gamma^*$, 
$$ v \mapsto \omega v\ \ (v \in B_r(G)), \ \ \text{and} \ \ w \mapsto \Omega  w  \ \ (w \in B_r(G \times G))$$ are completely contractive, we have that the given product,
\begin{equation}\label{eq:point-wise-multiplication}
 (\omega u , \omega v)  \mapsto \omega  \Gamma^*(\Omega  (u \otimes v))
\end{equation}
is completely contractive.

Now let $v\in A(G)$ and $ u \in P_r(G)$, the space of positive definite functions weakly associated to $\lambda$.  Then $v(s)=(\lambda_s\xi,\eta)$ for some $\xi, \eta \in L^2(G)$.
Also, since $u$ is  positive definite, there is a GNS construction $(\pi, \cK, \zeta)$ so that $\langle u, x\rangle = ( \pi(x)\zeta, \zeta )$ ($x\in C^*_r(G)$). Let $y\in C_r^*(G)$ be arbitrary.  Write also $\pi$ for the corresponding group representation of $G$ on $\cK$. One has, in particular, that  $\pi$ is weakly associated to $\lambda$ (see \cite{kaniuth-lau}).
 Using 
that $(y\otimes 1)\Gamma(x)\in C_r^*(G)\otimes C_r^*(G)$ and $\pi$ is nondegenerate, 
for each $x \in C^*_r(G)$, we have
\begin{eqnarray*}
\langle (\omega ( u\cdot y))(\omega v), x\rangle &=&  \langle (\omega  \otimes \omega ) (u\cdot y \otimes v), \Gamma(x) \rangle\\
&=& \langle \Gamma(\omega ) \Omega  (u\cdot y \otimes v), \Gamma(x)\rangle \\
&=& \langle   \Omega  (u\cdot y \otimes v), \Gamma(x\omega )\rangle \\
&=&\langle    (u \otimes v), (y\otimes 1)\Gamma(x\omega )\Omega  \rangle \\
&=& ( (\pi \otimes \text{id}) ((y\otimes 1)\Gamma(x\omega )\Omega  ) (\zeta \otimes \xi), \zeta \otimes \eta) \\
&=& ( (\pi \otimes \text{id}) (\Gamma(x\omega ))(\pi \otimes \text{id})(\Omega  ) (\zeta \otimes \xi), (\pi \otimes \text{id}) (y\otimes 1)(\zeta \otimes \eta))\\
&=& ( (\pi \otimes \text{id}) \circ \Gamma)(x\omega ) \zeta_1 , \zeta_2 ),
\end{eqnarray*}
where $\zeta_1,\zeta_2\in \cK \otimes L^2(G)$. 
By Fell's absorption principle,  we have that there is a  sum $\bigoplus_{i\in I} \lambda$ which is unitarily equivalent to the group representation $\pi \otimes \lambda $. Having  this equivalence, one can find vectors $\varsigma_1, \varsigma_2$ so that,
\[
 ( (\pi \otimes \text{id}) \Gamma(x\omega ))\zeta_1, \zeta_2) = \big(\big(\bigoplus_{i\in I} x\omega\big)  \varsigma_1, \varsigma_2 ).
 \]
But  the map
\[
z \mapsto \big(\big(\bigoplus_{i\in I}z) \varsigma_1, \varsigma_2 \big), \quad (z \in VN(G))
\]
is a  normal functional in $VN(G)_*=A(G)$.  Equivalently, there is some $\sigma \in A(G)$ so that
\[
\langle (\omega  (u\cdot y))(\omega v), x\rangle = \langle \sigma,  x\omega \rangle
\]
for every $x \in C^*_r(G)$.  Hence, $(\omega (u\cdot y))(\omega v) = \omega \sigma$ which is an element in $A(G, \omega )$.  Since every element of $B_r(G)$ can be written as a sum of four elements in $P_r(G)$, and by Cohen's factorization theorem $B_r(G)=B_r(G)\cdot C_r^*(G)$, we have that $A(G,\omega )$ is an ideal in $B_r(G, \omega )$.

To prove that $B_r(G,\omega )$ is faithful, suppose that $v \in B_r(G)$ is fixed so that for every $u \in A(G)$, $(\omega v)(\omega u)=0$.
For an arbitrary pair $u' \in B_r (G)$ and $y \in C^*_r(G)$, let $u$ be $u' \cdot y$ as an element in $B_r(G)$.
Therefore, for each $x \in C^*_r(G)$ we have
\begin{eqnarray*}
0 = \langle (\omega v)(\omega u) , x\rangle &=& \langle \omega  \Gamma^*(\Omega (v \otimes u)), x\rangle\\
&=& \langle    v \otimes u , \Gamma(x\omega )\Omega  \rangle\\
&=& \langle    v \otimes u , \Gamma(x)\Gamma(\omega )\Omega  \rangle\\
&=&\langle    v \otimes u , \Gamma(x) (\omega  \otimes \omega)  \rangle\\
&=&\langle    (v \otimes u') \cdot (1 \otimes y) , \Gamma(x) (\omega  \otimes \omega)  \rangle\\
&=&\langle    v \otimes u'   ,(1 \otimes y) \Gamma(x) (\omega  \otimes \omega ) \rangle.
\end{eqnarray*}
But we know that the space spanned by $\{(1 \otimes y) \Gamma(x) : x,y \in C^*_r(G)\}$ is dense in $C^*_r(G) \otimes C^*_r(G)$. Meanwhile, $\omega  \otimes \omega $, as an element in $M(C^*_r(G) \otimes C^*_r(G))$, has a dense range.  Therefore, the space spanned by $\{(1 \otimes y) \Gamma(x) (\omega  \otimes \omega)  : x,y \in C^*_r(G)\}$ is dense in $C^*_r(G) \otimes C^*_r(G)$. But if $v \neq 0$, then  there is some $x \in C^*_r(G)$ so that $\langle v, x\rangle \neq 0$; consequently, $\langle v \otimes v, x \otimes x \rangle \neq 0$ which is a contradiction.
\end{proof}

We note that similar to the duality $A(G, \omega )^*=VN(G)$, one can readily show that $B_r(G,\omega )$ is the dual space of $C^*_r(G)$ where for each $x \in C^*_r(G)$ and $u \in B_r(G)$,

\begin{equation}\label{eq:duality-B(G,omega)}
\langle \omega u, x\rangle_\omega :=\langle  u,  x \rangle 
\end{equation}
 given that the second product $\langle \cdot, \cdot\rangle$ is from the duality $(B_r(G), C^*_r(G))$.  Here we use $\cdot_\omega $ to denote the action of elements of $B_r(G, \omega )$ on an element $x \in C^*_r(G)$ i.e.
 \[
\langle \omega v, \omega  u \cdot_\omega x  \rangle_\omega  := \langle (\omega  u) (\omega v), x \rangle_\omega 
\]
for every $v \in B_r(G)$.

\begin{corollary}\label{c:C*.A(G,w)-dense-C*r}
Let $G$ be a locally compact group with a weight inverse $\omega $. Then the linear span of $\{\omega u \cdot_\omega x  : x \in C^*_r(G), u \in A(G)\}$ is dense in $C^*_r(G)$.
\end{corollary}

\begin{proof}
To show this, note that if $\omega v \in B_r(G,\omega )$
is such that $\langle \omega v,  \omega u\cdot_\omega x\rangle_\omega  = 0$ for all $u \in A(G)$ and $x \in C^*_r(G)$, then $(\omega u) (\omega v)=0$
for all $u \in A(G)$. 
By the last part of the proof of Theorem \ref{t:main}, 
$\omega v=0$.
\end{proof}

 \vskip1.0em

We recall that a Banach algebra $\mathcal{A}$ is a \emph{dual Banach algebra}  if $\mathcal{A} = (\mathcal{A}_* )^*$ for a closed $\mathcal{A}$-submodule $\mathcal{A}_*$ of $\mathcal{A}^*$. It is known that $B_r(G)$   for locally compact groups $G$ are forming dual Banach algebras. In the following we show that this is also the case for $B_r(G, \omega )$. 

\begin{lem}\label{l:B(G,omega)-dual-Banach-algebra}
Let $G$ be a locally compact group with a weight inverse $\omega $. Then $B_r(G, \omega )$ is a dual Banach algebra.
\end{lem}

\begin{proof}
To show that $B_r(G,\omega )$ is a dual Banach algebra, it suffices to show that the dual action  $\omega v \cdot_\omega  x$ ($x \in C^*_r(G)$ and $v \in B_r(G)$) yields an element in $C^*_r(G)$.   Recall that $C^*_r(G)^*=B_r(G)$ and therefore, $C^*_r(G)$ acts on $B_r(G)$ continuously. Based on this action and by Cohen's factorization theorem we have $v = v' \cdot y$ for some $y' \in C^*_r(G)$ and $v' \in B_r(G)$. Moreover, as $(1\otimes y)\Gamma(x)\in C_r^*(G)\otimes C_r^*(G)$, by Lemma \ref{mult}, $(1\otimes y)\Gamma(x)\Omega$ belongs to $C_r^*(G)\otimes C_r^*(G)$, and the slicing $\id \otimes v$ on $ C^*_r(G \times G)=C^*_r(G)\otimes C^*_r(G)$ maps the latter into $C^*_r(G)$.  Using this, we obtain 
 \begin{eqnarray*}
\langle \omega u, \omega v \cdot_\omega  x \rangle_\omega &=& \langle (\omega u)(\omega v), x\rangle_\omega \\
&=& \langle \omega  \Gamma^*(\Omega (u \otimes v)), x\rangle_\omega \\
&=&  \langle \Omega (u \otimes v) , \Gamma(x ) \rangle\\
&=&  \langle u \otimes v, \Gamma(x)\Omega  \rangle\\
&=&  \langle u \otimes (v'\cdot y), \Gamma(x)\Omega  \rangle\\
&=&  \langle (u \otimes v')\cdot (1 \otimes y), \Gamma(x )\Omega  \rangle\\
&=&  \langle (u \otimes v') , (1 \otimes y)\Gamma(x) \Omega  \rangle\\
&=&  \langle u, \id \otimes v' \left((1\otimes y)\Gamma(x)\Omega  \right) \rangle \\
&=&  \langle \omega u, \id \otimes v' \left((1 \otimes y)\Gamma(x) \Omega  \right) \rangle_\omega .
 \end{eqnarray*}
Since $u \in B_r(G)$ was arbitrary, it gives 
\begin{align*}
 \omega v \cdot_\omega  x = \id \otimes v' \left((1 \otimes y)\Gamma(x) \Omega  \right)\in C_r^*(G).  
\end{align*}
\end{proof}

 \section{Weighted generalization of the Herz restriction theorem}

To lift weights from subgroups we apply the theory of duality for subgroups developed in \cite{tata}. Here we mainly use the modern treatment of the topic in the language of locally compact quantum groups in \cite{ds}. Lifting of weights from subgroups to the full group has been considered and studied in the literature, see \cite{gllst, LeSa, lst}. However, in this section, we develop a comprehensive approach based on inverse weights.


Let $H$ be a closed subgroup of $G$. Then by \cite[Theorem~3.3]{ds} and its proof (see also \cite[Proposition 2.6.6]{kaniuth-lau}), there exists a normal injective $*$-homomorphism $\gamma:VN(H) \rightarrow VN(G)$ such that
\begin{equation}\label{eq:gamma}
\gamma(\lambda^{H}_s)=\lambda_s^G
\end{equation}
 where $\lambda^G$ and $\lambda^H$ denote the left regular representations of $G$ and $H$, respectively. 
Furthermore, the restriction of $\gamma$ to $C^*_r(H)$, still denoted by $\gamma$, is a non-degenerate $*$-homomorphism $\gamma: C^*_r(H) \rightarrow M(C^*_r(G))$. In fact, using Cohen's factorization theorem, $\gamma(C^*_r(H)) C^*_r(G)$ and $C^*_r(G) \gamma(C^*_r(H))$ are equal to $C^*_r(G)$.
Let us again use $\gamma$ to show the unique extension $\gamma$ as a $*$-homomorphism from $M(C^*_r(H))$ into $M(C^*_r(G))$.

Note that since $H$ is a closed subgroup of $G$, the same holds for $H \times H$ with respect to $G \times G$. Then using \eqref{eq:gamma}, we have that
\begin{equation}\label{eq:gamma-x-gamma}
\gamma \otimes \gamma: VN(H)\bar\otimes VN(H) \rightarrow VN(G) \bar\otimes VN(G)
\end{equation}
 and by \cite[Theorem~3.3 (3)]{ds},
 \begin{equation}\label{eq:Gamma}
 (\gamma \otimes \gamma) \circ \Gamma_H= \Gamma_G \circ \gamma
 \end{equation}
where $\Gamma_G$ and $\Gamma_H$ are the co-multiplication of $VN(G)$ and $VN(H)$, respectively.

\begin{proposition}\label{p:lifting-weights}
Let $G$ be a locally compact group with a closed subgroup $H$. If $\omega $ is a weight inverse for $H$, then $\gamma(\omega )$ is a weight inverse for $G$. 
\end{proposition}

\begin{proof}
As we saw before, $\gamma(\omega )$ is an element in $M(C^*_r(G))$ with $\norm{\gamma(\omega )} \leq \norm{\omega }\leq 1$. Let $\epsilon>0$ and $x \in C^*_r(G)$ be given. Since $C^*_r(G) = \gamma(C^*_r(H)) C^*_r(G)$,  there is a decomposition $x =\gamma(y) x'$ for $y \in C^*_r(H)$ and $x' \in C^*_r(G)$. Since $\omega  C^*_r(H)$ is dense in $C^*_r(H)$, there is some $y' \in C^*_r(G)$ so that $\norm{\omega y' - y}<\epsilon/\norm{x'}$.  Therefore,
\[
\norm{ x - \gamma(\omega ) \left( \gamma(y')x'\right) } \leq \norm{y - \omega  y'} \norm{x'} < \epsilon.
\]
Therefore, $\gamma(\omega ) C^*_r(G)$ is dense in $C^*_r(G)$. A similar argument implies that  $C^*_r(G) \gamma(\omega )$ is dense in $C^*_r(G)$.

Note that since $H \times H$ is a subgroup of  $G \times G$, we have the following restriction from \eqref{eq:gamma-x-gamma}
 \[
 \gamma \otimes \gamma: M( C^*_r(H) \otimes C^*_r(H))  \rightarrow  M(C^*_r(G) \otimes C^*_r(G)).
 \]
 Subsequently, using \eqref{eq:Gamma} and condition $(2)$ on $\omega$ from Definition \ref{D:Weight}, we get that
 \begin{eqnarray*}
\gamma(\omega ) \otimes \gamma(\omega ) &=& (\gamma \otimes \gamma)( \omega \otimes \omega )\\
&=& (\gamma \otimes \gamma)(\Gamma_H (\omega )\Omega )\\
&=& \Gamma_G( \gamma(\omega )) (\gamma\otimes \gamma)(\Omega ).
 \end{eqnarray*}
 Hence, $\gamma(\omega )$ is a weight inverse for $G$.
\end{proof}

Since $\gamma$ is a normal injective $*$-homomorphism, its adjoint restricted to $A(G)$ forms a mapping $\gamma_*: A(G) \rightarrow A(H)$.
As $\gamma$ is injective and hence an isometry, $\gamma_*$  is a surjection, (see \cite[Theorem~3.7]{ds}). In fact, in the light of \eqref{eq:Gamma}, it is easy to show that $\gamma_*(u)$ is nothing but the well-known Herz restriction of elements of $A(G)$ to $A(H)$. Here we give a generalization of this result to weighted Fourier algebras. 

\begin{proposition}\label{t:weighted-resteriction}
Let $G$ be a locally compact group with a closed subgroup $H$. For each weight inverse $\omega $ on $H$,
the restriction map $u \mapsto u|_H$ maps $A(G,\gamma(\omega ))$ onto $A(H,\omega )$.
\end{proposition}

\begin{proof}
Let  $u \in A(G)$. Therefore, $\gamma(\omega )u \in A(G, \gamma(\omega ))$. Note that for each $s \in H$ and $u \in A(G)$, by \eqref{eq:point-evaluation}, we have
\begin{eqnarray*}
\gamma(\omega )u(s) &=& \langle \gamma(\omega )u,  \lambda^G_s,\rangle\\
&=&\langle u, \gamma(\lambda_s^H) \gamma(\omega )\rangle\\
&=&\langle \gamma_*(u),  \lambda_s^H \omega \rangle\\
&=&\langle \omega \gamma_*(u),  \lambda_s^H \rangle.
\end{eqnarray*}
 Therefore, 
 $\gamma(\omega )u |_H =\omega  \gamma_*(u)$ for every $u \in A(G)$.  Since $\{\gamma_*(u): u \in A(G)\}$ is equal to $A(H)$ we have that the restriction $\gamma(\omega ) u \mapsto \gamma(\omega ) u |_H$ is   onto.
\end{proof}

\begin{eg}\label{eg:semidirect}
Let $N$ be an abelian locally compact group and  $H$ be a locally compact subgroup of automorphism  $N$. Let $G:=N\rtimes H$ be the semidirect product of $N$ and $H$. We may identify $N$ with the subgroup $\{(t,e): t\in N\}$ of $G$ where $e$ is the identity of the group $H$. Let $\alpha: H\to \text{Aut}(N)$ be the map implementing the semidirect product and let $\delta_H: H\to \mathbb R^+$ be a continuous homomorphism  such that $\int_Nf(x)dm_N(x)=\delta(h)\int_Nf(\alpha_h(x))dm_N(x)$ for any measurable function $f$ on $N$. 

Here we   apply Proposition~\ref{t:weighted-resteriction} to lift a  weight inverse from $N$ to $G$. To do so, let $\xi \in L^2(G)$, for each $t \in N$, first note that by \eqref{eq:gamma},
\[
\gamma(\lambda^N_t)  \xi(s,r)= \lambda^G_{(t,e)} \xi (s,r) = \xi(t^{-1} s,r), \ \ \ \ (s,r) \in G
\]
 Since $dm_G(s,r)  = dm_N(s)\delta_H^{-1}(r) dm_H(r) $, 
 $$L^2(G)= L^2(N,dm_N)\otimes L^2(H, \delta_H^{-1} dm_H).$$  Therefore,
 \begin{equation}\label{eq:M(A)-on-L^2(G)}
\gamma(\lambda^N_t)  \xi =  (\lambda^N_t\otimes \id) \xi.
\end{equation}
Let $w $ be a continuous submultiplicative function on $\hat{N}$ i.e. $w (ts) \leq w(s) w (t)$ ($s,t \in \hat{N}$)
and assume  $1/w  \in C_b(\hat{N})$. 
Let $\omega $ denote its Fourier transform on locally compact abelian group $N$, i.e. $\omega  \eta = {\cal F}_N^{-1}(\frac{1}{w }   {\cal F}_N(\eta))$, for every $\eta \in L^2(N)$, where  ${\cal F}_N$ and ${\cal F}^{-1}_N$ denote the Fourier transform and its inverse for  $N$.  Consequently, $\omega  \in M(C^*_r(N))$ and subsequently  $$\gamma(\omega ) \in \gamma(M(C^*(N))) \subseteq VN(G).$$
Since $\omega $ belongs to $VN(N)$ and the linear span of $\{\lambda^N_s: s\in N\}$ is dense in $VN(N)$, there is a net $(x_\alpha)_\alpha$ in the linear span of $\{\lambda^N_s: s\in N\}$ so that $x_\alpha \rightarrow \omega $ weakly$^*$. But   weak$^*$-weak$^*$  continuity of  $\gamma$ implies that $\gamma(x_\alpha) \rightarrow \gamma(\omega )$.  So applying a linear extension of \eqref{eq:M(A)-on-L^2(G)}, we get 
\begin{eqnarray*}
(\gamma(\omega ) \xi_1 \otimes \eta_1 , \xi_2 \otimes \eta_2) &=& \lim_\alpha (\gamma(x_\alpha) \xi_1 \otimes \eta_1 , \xi_2 \otimes \eta_2)\\
&=& \lim_\alpha ( (x_\alpha\otimes \id)\xi_1 \otimes \eta_1 , \xi_2 \otimes \eta_2)\\
&=& ( (\omega  \otimes \id)\xi_1 \otimes \eta_1 , \xi_2 \otimes \eta_2)\\
&=& ( {\cal F}_N^{-1}(\frac{1}{w }   {\cal F}_N(\xi_1)) \otimes \eta_1 , \xi_2 \otimes \eta_2),
\end{eqnarray*}
where $\xi_1, \xi_2 \in L^2(N)$ and $\eta_1, \eta_2 \in L^2(H, \delta_H^{-1} dm_H)$.
\end{eg}

\begin{eg}\label{weight_functions}
For $a$, $s$, $t\geq 0$, $0\leq b\leq 1$, let
$$w_{a,b,s,t}(x)=e^{a|x|^b}(1+|x|^s)\ln(e+|x|)^t, x\in\mathbb R.$$
Then $w_{a,b,s,t}$ is a continuous submultiplicative function on $\mathbb R$ and $\mathbb Z$ with $1/w\in C_b(\mathbb R)$.  Hence the corresponding Fourier transforms give weight inverses $\omega\in M(C_r^*(\mathbb R))$ and $C_r^*(\mathbb T)$ respectively. 
\end{eg}
\begin{cor}\label{c:R-T-subgroups}
Let $G$ be a locally compact group with a closed subgroup isomorphic to $\Bbb{R}$ or $\Bbb{T}$. Then $G$ possesses non-invertible weight inverses.
 \end{cor}

 \begin{proof}
 Note that for $H$ isomorphic to $\Bbb{R}$ or $\Bbb{T}$, Example \ref{weight_functions} provides a family of  non-invertible weight inverses $\omega $.   By  Corollary \ref{c:omega-invertible}, the inclusion $A(H,\omega ) \subseteq A(H)$ is proper. Now if $\gamma(\omega )$ is invertible, by the same corollary
 $A(G,\gamma(\omega ))= A(G)$. Consequently,  $A(G,\gamma(\omega ))|_H=A(G)|_H$. But this implies that $A(H,\omega )=A(H)$ which is a contradiction.
 \end{proof}


\section{Completely bounded multipliers of $A(G,\omega )$}\label{s:McbA(G,w)}

Let $G$ be a locally compact group with a weight inverse $\omega $.  Classically, it is known that every element of $MA(G,\omega )$ can be considered as a bounded continuous function on the maximal ideal space of $A(G,\omega )$. But by Theorem \ref{l:G->D(A(G,w))}, we can restrict every element of $MA(G,\omega )$ to a bounded continuous function on $G$. As  $A(G,\omega )$ is a dense subspace of $A(G)$, the restriction of $MA(G)$ to $G$ is injective, i.e. if there are two distinct element in $MA(G,\omega )$, their restrictions to $G$ are also distinct. 
In the light of this observation, from now on we identify every element of $MA(G,\omega )$, and in particular elements of $M_{\rm cb}A(G,\omega )$, the space of completely bounded multipliers on $A(G,\omega)$, with their (bounded continuous) restriction to $G$.
Recall that $A(G,\omega )^*=VN(G)$ and hence has an operator structure inherited from $VN(G)$. In particular, $\bm_\phi\in M_{\rm cb}A(G,\omega )$ if $\bm_\phi^*:VN(G)\to VN(G)$ is completely bounded. 

In this section and the next one, we investigate $M_{\rm cb}(A(G,\omega))$ and some of its natural properties. We start with the following which is a generalization of a characterization of $M_{\rm cb}(A(G))$ to the weighted case from \cite[Theorem 1.6]{canniere-haagerup}. The proof is a modification of the original one and is presented here for the sake of completion.  

\begin{proposition}\label{p:cb-multiplier-A(g,w)}
Let $G$ be a locally compact group with a weight inverse $\omega$ and let  $\phi$ be a function of $G$. Then the following are equivalent.
\begin{itemize}
\item[(i)]{$\bm_\phi \in M_{\rm cb}A(G,\omega )$.}
\item[(ii)]{$\phi \times 1_{SU(2)}$ is a multiplier of $A(G\times SU(2), \omega  \otimes 1)$.}
\end{itemize}
\end{proposition}

\begin{proof}
Recall that $VN(SU(2))= \bigoplus_{k \in \Bbb{N}} \Bbb{M}_k(\Bbb{C})$.  Therefore,
\begin{eqnarray*}
\left(A(G,\omega ) \hat{\otimes} A(SU(2))\right)^* &\cong & CB( A(G,\omega ), VN(SU(2)))\\
&\cong & CB( A(G,\omega ), \bigoplus_{k \in \Bbb{N}} \Bbb{M}_k(\Bbb{C}) )\\
&\cong &  \bigoplus_{k \in \Bbb{N}}  CB( A(G,\omega ),\Bbb{M}_k(\Bbb{C}) )\\
&\cong &  \bigoplus_{k \in \Bbb{N}}  \Bbb{M}_k( VN(G)  ).
\end{eqnarray*}
$(i) \Rightarrow (ii)$.  If $\bm_\phi \in M_{\rm cb}A(G,\omega )$, therefore
\begin{eqnarray*}
\norm{\phi\times 1_{SU(2)}}_{MA(G\times SU(2), \omega  \otimes 1)} &=& \norm{\bm_{\phi \times 1_{SU(2)}}^* }_{\mathcal{B}(VN(G) \bar{\otimes} VN(SU(2)) )}\\
&=&  \sup_{k \in \Bbb{N}} \norm{\bm_\phi \otimes \id_k} = \norm{ \bm_\phi}_{M_{\rm cb}A(G,\omega )}.
\end{eqnarray*}
$(ii) \Rightarrow (i)$. For each $k_0 \in \Bbb{N}$, note that
\[
\norm{\bm_\phi^* \otimes \id_{k_0}} \leq \norm{ \bm_{\phi \times 1_{SU(2)}}^*}_{\mathcal{B}(VN(G) \bar{\otimes} VN(SU(2)))}
= \norm{ \bm_{\phi \times 1_{SU(2)}}}_{MA(G \times SU(2), \omega \otimes 1)}< \infty.
\]
Hence, $\norm{\bm_\phi }_{M_{\rm cb}A(G, \omega )} <\infty$.
\end{proof}

The following statement shows that $B_r(G,\omega )$ is a subalgebra of  $M_{\rm cb}A(G,\omega )$.

\begin{cor}\label{c:Br(G,w)-cb-multipliers}
Let $G$ be a locally compact group and $\omega $ be a weight inverse on $G$. Then $B_r(G,\omega )$ is injectively embedded into $M_{\rm cb}A(G,\omega )$.
\end{cor}

\begin{proof}
Fix $\omega u$ for $u  \in B_r(G)$. Then $u(s)=(\pi(s)\xi,\eta )$, $s\in G$, for some representation $\pi:G\to B(H_\pi)$, with $\xi, \eta\in H_\pi$, which extends to a non-degenerate $*$-representation of $C^*_r(G)$ to $B(H_\pi)$. Since $\pi$ then extends to a $*$-representation of the multiplier algebras $M(C^*_r(G))$, also denoted by 
$\pi$, we have  $\omega u(s)=( \pi(s) \pi(\omega )\xi,\eta )$, $s\in G$.  Write $1_{SU(2)}$ for the constant function 1 on $SU(2)$. Therefore,  for each pair $(s,t) \in G \times SU(2)$, 
\begin{eqnarray*}
(\omega u \otimes 1_{SU(2)}) (s,t) &=& ( \pi(s) (\pi(\omega )\xi), \eta  )\\
&=& ( (\pi\times \lambda^{SU(2)} )(s,t) \left(\pi(\omega )\xi \otimes 1_{SU(2)}\right), \eta \otimes 1_{SU(2)} )\\
&=& ( (\pi\times \lambda^{SU(2)} )(s,t)(\pi\otimes \text{id})(\omega \otimes 1) \left(\xi \otimes 1_{SU(2)}\right), \eta \otimes 1_{SU(2)})\\
&=& (\omega \otimes 1) ( (\pi\times \lambda^{SU(2)})(s,t) \left(\xi \otimes 1_{SU(2)}\right), \eta \otimes 1_{SU(2)} ),
\end{eqnarray*}
where $ (\omega \otimes 1) ( (\pi\times \lambda^{SU(2)})(s,t) \left(\xi \otimes 1_{SU(2)}\right), \eta \otimes 1_{SU(2)})$ is an element in $B_r(G \times SU(2), \omega  \otimes 1)$. By Theorem~\ref{t:main}, we know that this is a  Banach algebra with $A(G \times SU(2), \omega  \otimes 1)$ as an ideal in it. Therefore, $B_r(G \times SU(2), \omega  \otimes 1)$ sits  in $MA(G \times SU(2), \omega  \otimes 1)$. So, by Proposition~\ref{p:cb-multiplier-A(g,w)}, $\omega u$ is in $M_{\rm cb}A(G,\omega )$. By Theorem~\ref{t:main}, we know that $B_r(G,\omega )$ is a faithful Banach algebra, therefore, a non-zero $\omega u$ cannot be a zero multiplier; that is the inclusion is injective.
\end{proof}

\medskip

Let $\phi$ be a multiplier of $A(G, \omega )$ which induces the multiplication map $\bm_\phi:A(G, \omega ) \rightarrow A(G, \omega )$. Hence
\[
\bm_\phi( (\omega u)(\omega v))= \bm_\phi(\omega u) (\omega v) = (\omega u) \bm_\phi(\omega v).
\]
Since the dual space of $A(G, \omega )$ is $VN(G)$, we have that the map $\bm^*_\phi: VN(G)\rightarrow VN(G)$ is given by
\[
\langle \omega u, \bm^*_\phi(x)\rangle_\omega  := \langle \bm_\phi (\omega u), x \rangle_\omega,  \quad \quad  u\in A(G).
\]

When $G$ is abelian, then $M_{\rm cb}(A(G,\omega))=M(A(G,\omega))$ and by \cite{gaudry}, it can be identified, via the Fourier transform with the weighted measure algebra $M(w)\subset M(\widehat G)$ on $\widehat{G}$, here $M(\widehat{G})$ is the measure algebra of $\widehat G$ and $w$ is a continuous submultiplicative function on $\widehat G$ related to $\omega$: $|\omega|=\mathcal F^{-1}M_{1/w}\mathcal F$.  In particular, $M_{\rm cb}(A(G,\omega))$ becomes a subspace of $M_{\rm cb}(A(G))$. We will show below that this embedding also hold for a large class of non-abelian groups.  


\medskip

By \cite{boz, jol}, completely bounded multipliers of $A(G)$ are related to Schur multipliers. So we start by recalling the notion of measurable Schur multipliers. Let $(X,\mu)$ be a measure space and $\varphi\in L^\infty(X\times X,\mu\times \mu)$. We let 
$$S_\varphi: L^2(X\times X, \mu\times \mu)\to L^2(X\times X,\mu\times\mu), k\mapsto\varphi k.$$
We may identify this map with a (densely defined) linear map on $K(L^2(X,\mu))$, the space of compact operators on $L^2(X,\mu)$, 
acting by $S_\varphi(I_k)=I_{\varphi k}$, where $I_k$ is the Hilbert-Schmidt kernel operator given by $(I_kg)(y)=\int_X k(x,y)g(x)d\mu(y)$. 
We say that $\varphi$ is a (measurable) Schur multiplier with respect to $(X\times X,\mu\times \mu)$ or simply on $X$ if the measure is clear, if $S_\varphi$ extends to a bounded map on the whole $K(L^2(X,\mu))$. It has then a unique weak* continuous extension to $B(L^2(X,\mu))$ which will be still denoted by $S_\varphi$. We note that $S_\varphi$ is automatically completely bounded.  By \cite{peller} (see also \cite{spronk, tt_survey}),   $\varphi$ is a Schur multiplier  on $X$ if and only if there exists a Hilbert space $\cH$ and {\it bounded} measurable functions $\xi, \eta$ from $X$ to $\cH$ such that $\varphi(x,y) = (\xi(x), \eta(y))_{\cH}$ for almost all $x, y$ in $X$, and the space of Schur multipliers  can be identified with the weak*-Haagerup tensor product $L^\infty(X,\mu)\otimes^{w^*h}L^\infty(X,\mu)$, see \cite{blecher-smith}  for the definition. The norm, $\|\varphi\|$, of the Schur multiplier $\varphi$ is by definition the norm of the map $S_\varphi$ and equal to $$\inf\{\|\xi\|_{L^\infty(X,\cH)}\|\eta\|_{L^\infty(X,\cH)}: \varphi(x,y)=(\xi(x),\eta(y))_{\cH} \text{ a.e.}\}.$$ 

\begin{proposition}\label{P:weighted cb mult-Schur mult}
Let $G$ be a locally compact group, and let $\omega$ be a weight inverse on it. Suppose that $\phi \in M_{\rm cb}A(G,\omega )$. Then 
for every $x\in M(C^*_r(G))$ and $\zeta$ and $\theta$ in $L^2(G)$, the mapping 
\begin{align}\label{Eq:weighted multiplier as Schur Mult}
    (t,s)\mapsto \phi(ts^{-1})(\lambda_s^*\omega x \lambda_s\zeta,\theta) \ , \ G\times G \to \mathbb{C},
\end{align}
is a Schur multiplier on $G$.  
\end{proposition}

\begin{proof}
Note first that by Lemma~\ref{l:G->D(A(G,w))}, for each $s \in G$, the point evaluation at $s$  belongs to $A(G,\omega )^*$. Since $\bm^*_\phi$ is a completely bounded mapping from $VN(G)$ to itself, by  Wittstock's factorization theorem 
(\cite[Theorem 8.4]{paulsen}), there exists a Hilbert space $\cH$, a nondegenerate $*$-representation $\pi: VN(G) \rightarrow B({\cH})$ and two bounded operators $V_1, V_2$ from $L^2(G)$ to ${\cH}$ such that $\bm_\phi^*(x) = V_2^* \pi(x)V_1$ for $x  \in VN(G)$ and $\norm{\bm_\phi}_{M_{\rm cb}A(G, \omega )} = \norm{V_1}\norm{V_2}$. Therefore, for every $s \in G$ using \eqref{eq:point-evaluation}, we get
\[
\phi(s) \lambda_s \omega  = \bm_\phi^*(\lambda_s \omega )  = V_2^* \pi(\lambda_s\omega  ) V_1 = V_2^* \pi(\lambda_s)\pi(\omega  ) V_1.
\]
and
\[
\phi(ts^{-1}) \lambda_t\lambda_{s}^* \omega   = V_2^* \pi(\lambda_t)\pi(\lambda_s)^*\pi(\omega  ) V_1.
\]
Let $x\in M(C^*_r(G))$ and $\zeta$ and $\theta$ be vectors in $L^2(G)$. Then
\begin{eqnarray*}\phi(ts^{-1})(\lambda_s^*\omega x \lambda_s\zeta,\theta)=(\lambda_t^*V_2^* \pi(\lambda_t)\pi(\lambda_s)^*\pi(\omega  ) V_1 x\lambda_s\zeta,\theta)\\
=(\pi(\lambda_s)^*\pi(\omega  ) V_1 x\lambda_s\zeta, \pi(\lambda_t)^*V_2\lambda_t\theta).
\end{eqnarray*}
Let $\xi(s)=\pi(\lambda_s)^*\pi(\omega  ) V_1 x\lambda_s\zeta$ and
$\eta(t)=\pi(\lambda_t)^*V_2\lambda_t\theta$. Then, both $\xi$ and $\eta$ are bounded with
$$\phi(ts^{-1})(\lambda_s^*\omega x \lambda_s\zeta,\theta)=(\xi(s),\eta(t))_{\cH}.$$
Hence the mapping in \eqref{Eq:weighted multiplier as Schur Mult} is a Schur mutiplier. 

\end{proof}

We say that $u:G\to\mathbb C$ belongs to  $A(G)$ locally (write $u\in A(G)^{\rm loc}$) if for any $t\in G$ there is a neighbourhhood $U$ of $t$ and $v\in A(G)$ such that $u(s)=v(s)$ for any $s\in U$.

Let $A$ be a unital $G$-$C^*$-algebra, i.e. $A$ is a unital $C^*$-algebra and $G$ acts continuously on $A$ by a $*$-automorphism. For every $a\in A$, we let $D_{A,G}(a)$ be the closure of the convex hull of the orbit $\{g\cdot a: g\in G \}$ in $A$.

\begin{theorem}\label{mcb}
Let $G$ be a locally compact groups and $\omega $ be a weight inverse on it. Then the following hold:
\begin{enumerate}
    \item 
$M_{\rm cb}A(G,\omega )$ is a subspace of $A(G)^{\rm loc}$; 
\item $M_{\rm cb}A(G,\omega )$ is a subspace of $M_{\rm cb}A(G)$ if there exists $x\in C^*_r(G)$ such that  $0\notin D_{M(C^*_r(G)),G}(x)$. 
\end{enumerate}
\end{theorem}

\begin{proof}
First note that, by \cite[Proposition~2.2]{ors}, $|\omega^*|$ is also a weight inverse that admits the same Beurling-Fourier algebra $A(G,\omega )$. So, without loss of generality, we assume that $\omega $ is a positive element in $VN(G)$.



(i)
By the preceding proposition, for every 
$\zeta$ and $\theta$ in $L^2(G)$, the mapping 
\begin{align*}
   \psi: (t,s)\mapsto \phi(ts^{-1})(\lambda_s^*\omega  \lambda_s\zeta,\theta) \ , \ G\times G \to \mathbb{C},
\end{align*}
is a Schur multiplier on $G$.  Let $\psi(t,s)=(\xi(s),\eta(t))_{\cH}$, $s,t\in G$, be a representation of $\psi$ with $\xi$, $\eta\in L^\infty(G,\cH)$. 
As $\ker\omega =\{0\}$ (by Lemma \ref{l:ker(w-1)}) and hence $\ker \omega^{1/2}=\{0\}$, we have 
$$(\lambda_s^*\omega \lambda_s\zeta,\zeta)=\|\omega^{1/2}\lambda_s\zeta\|^2\ne 0 \ \ \ (s\in G).$$
Let $\tilde\xi(s)=\|\omega^{1/2}\lambda_s\zeta\|^{-2}\xi(s)$, $s\in G$. 
Then
$$\phi(ts^{-1})=(\tilde\xi(s),\eta(t))_{\cH}.$$
As $\|\omega^{1/2}\lambda_s\zeta\|^{-2}$  is continuous and hence bounded on every compact subset $K\subset G$, $\phi|_{K\times K}$ is a Schur multiplier on $K$ for all such $K\subset G$. By \cite[Remark 7.6, Lemma 7.4]{stt_closable},
 we obtain $\phi\in A(G)^{\rm loc}$.

(ii) We first note that the relation \eqref{Eq:weighted multiplier as Schur Mult} implies that for every $x\in C^*_r(G)$ and $u\in \mathcal{T}(L^2(G))$, the trace-class operators on $L^2(G)$, the mapping $\Phi_u:G\times G \to \mathbb{C}$ given by
\begin{align*}
    \Phi_u(t,s)=\phi(ts^{-1})\langle \lambda_s^*\omega x\lambda_s , u\rangle \ \ \ (t,s\in G),
\end{align*}
is a Schur multiplier with
\begin{align*}
    \|\Phi_u\|\leq C \|u\|_{\mathcal{T}(L^2(G))}.
\end{align*}
Now let $S\in \mathcal{T}(L^2(G))^{**}=B(L^2(G))^*$. Take a bounded net of $\{S_i\}_{i\in I} \subset \mathcal{T}(L^2(G))$ such that $S_i$ converges to $S$ in the weak$^*$-topology of $B(L^2(G))^*$. Since the space of Schur multipliers $L^\infty(G)\otimes^{w^*h} L^\infty(G)$ is a dual space, by taking a subnet if necessary, we can assume that $\Phi_{S_i}$ also has a weak$^*$-limit in $L^\infty(G)\otimes^{w^*h} L^\infty(G)$; we denote the limit point as $\Phi_S$. Putting all these together, we obtain that for
every $S\in \mathcal{T}(L^2(G))^{**}=B(L^2(G))^*$, the mapping $\Phi_S:G\times G \to \mathbb{C}$ given by
\begin{align*}
    \Phi_S(t,s)=\phi(ts^{-1})\langle S, \lambda_s^*\omega x \lambda_s \rangle \ \ \ (t,s\in G),
\end{align*}
is a Schur multiplier with
\begin{align*}
    \|\Phi_S\|\leq C \|S\|_{B(L^2(G))^*}.
\end{align*}
Now suppose that (we prove this below) for some $x\in C^*_r(G)$, $0 \notin D_{M(C^*_r(G)),G}(\omega x)$. Since $D_{M(C^*_r(G)),G}(\omega x)$ is a closed convex set in $B(L^2(G))$, it follows from the Hahn-Banach theorem that there is $S\in B(L^2(G))^*$ and $\epsilon>0$ satisfying 
\begin{align*}
   \text{Re}\ \langle S, T\rangle \geq \epsilon \ \ \ \forall T\in D_{M(C^*_r(G)),G}(\omega x).
\end{align*}
However, $D_{M(C^*_r(G)),G}(\omega x)$ contains every element $\lambda^*_s \omega x \lambda_s $, for all $s\in G$. This implies that 
\begin{align*}
    |\langle S, \lambda^*_s \omega x \lambda_s \rangle| \geq \epsilon \ \ \ \forall s\in G.
\end{align*}
In particular, the function $\psi:G\to \mathbb{C}$ given by $\psi(s)=\langle S, \lambda^*_s \omega x\lambda_s \rangle^{-1}$, $s\in G$, belongs to $L^\infty(G)$ so that  $(\psi\otimes 1)\Phi_S$ is a Schur multiplier. This completes the proof, since 
\begin{align*}
    \phi(ts^{-1})=[(\psi\otimes 1)\Phi_S](t,s) \ \ (t,s\in G).
\end{align*}
Finally, if for every $x\in C^*_r(G)$, $0 \in D_{M(C^*_r(G)),G}(\omega x)$, then it follows from the density of $\om C^*_r(G)$ in $C^*_r(G)$ that $0 \in D_{M(C^*_r(G)),G}(x)$ for all $x\in C^*_r(G)$. This contradicts the assumption of part (ii). Hence, for some $x\in C^*_r(G)$, we have $0 \notin D_{M(C^*_r(G)),G}(\omega x)$.  \end{proof}

\begin{rem}\label{rem_weight}
    Inspecting the proof of Theorem \ref{mcb}, one can easily see that the condition $0\notin D_{M(C_r^*(G)), G}(\omega)$ would already imply the inclusion $M_{\rm cb}A(G,\omega)\subset M_{\rm cb}A(G)$. Clearly the condition holds in the case of the trivial weight $\omega=I$ or more generally when $\omega$ is in the center of $M(C_r^*(G))$, while the condition $0\notin D_{M(C_r^*(G)),G}(x)$ can be violated for all $x\in C_r^*(G)$, see Theorem \ref{cb_Lie}. 
\end{rem}

\begin{corollary}
Let $G$ be a locally compact group and $\omega $ be a positive weight inverse on it which satisfies one of the following conditions:
\begin{itemize}
\item $\omega $ is central, i.e. $\omega $ is a central element of $M(C_r^*(G))$;
\item there exists a sub-representation $\pi$ of $\lambda$  such that $\pi(\omega)\geq\varepsilon$ for some $\varepsilon >0$. 
\end{itemize}
 Then $M_{\rm cb}A(G,\omega )$ is a subspace of $M_{\rm cb}A(G)$.
\end{corollary}

\begin{proof} The statement follows from Theorem \ref{mcb} and Remark \ref{rem_weight} by noting that
if $\omega $ is central then $\lambda_s^*\omega \lambda_s=\omega $ and hence $0\notin D_{M(C_r^*(G)), G}(\omega)$.
If $\omega $ satisfies the second condition, then there exists a subspace $U\subset L^2(G)$, invariant with respect to $\lambda$ and such that $(\omega \xi,\xi)\geq \varepsilon$ for any $\xi\in U$, $\|\xi\|=1$. Hence $\sup_{s\in G}(\omega \lambda_s\zeta,\lambda_s\zeta)^{-1}\leq\varepsilon^{-1}$ for any $\zeta\in U$, $\|\zeta\|=1$.  
\end{proof}

Next we  provide a large class of groups for which $M_{\rm cb}(A(G,\omega))$ embeds into Herz-Schur multipliers for any weight $\omega$. 

\begin{corollary}\label{p:MA(G,w)<MA(G)-trace}
Let $G$ be a locally compact group having an open amenable radical, and let $\omega $ be a weight inverse on it. Then $M_{\rm cb}A(G,\omega )$ is a 
subspace of $M_{\rm cb}A(G)$. This, in particular, if $G=G_1\times G_2$, where $G_1$ is amenable and $G_2$ is an [IN]-group. 
\end{corollary}

\begin{proof}
Since $G$ has open amenable radical, $C^*_r(G)$ has a tracial state (\cite{KenRuam, fsw}). Let $\tau$ be its (unique) extension to the multiplier algebra $M(C^*_r(G))$. 
Pick $x\in C^*_r(G)$ such that $\tau(x)\neq 0$. Since $\tau$ is a trace, it follows that $\tau$ is a nonzero constant on $D_{M(C^*_r(G)),G}(x)$. This, in particular, implies that $0\notin D_{M(C^*_r(G)),G}(x)$. Thus, by Theorem \ref{mcb}, $M_{\rm cb}A(G,\omega )$ is a subspace of $M_{\rm cb}A(G)$. The final result follows since every [IN]-group has an open normal amenable subgroup.
\end{proof}

\begin{rem}
(i) In Section \ref{sec_amen}, we show that for amenable groups, an even stronger statement holds. Namely,  if $G$ is amenable, then $M_{\rm cb}A(G,\omega)=B_r(G,\omega)$, as in the non-weighted case.

(ii) Let $G$ be a locally compact group. As stated (partially) in Corollary \ref{p:MA(G,w)<MA(G)-trace}, it is shown in \cite{KenRuam} that $C^*_r(G)$ has a trace if and only if $G$ has an open normal amenable subgroup. Although $G$ not having such a subgroup prevents us from applying Corollary \ref{p:MA(G,w)<MA(G)-trace}, one might still hope that the hypothesis of Theorem \ref{mcb}(ii) would stay valid, at least in some cases, as it looks like a weaker assumption than the existence of the trace.  In the following theorem, we show that it is not the case for many connected groups, that is there is a limitation on how far we can apply Theorem \ref{mcb}(ii). 
\end{rem}

\begin{theorem}\label{cb_Lie}
    Let $G$ be a non-compact, connected semisimple  Lie group with finite center. Then   
   $0\in D_{M(C^*_r(G)),G}(x)$ for every $x\in C^*_r(G)$.
    \end{theorem}

\begin{proof}
Suppose otherwise, i.e. there is $x\in C^*_r(G)$ such that $0\notin D_{M(C^*_r(G)),G}(x)\subset C^*_r(G)$. Since, $C_c(G)$, the space of continuous functions on $G$ with compact support, is dense in $C^*_r(G)$ and for every $s\in G$, $g,h\in C^*_r(G)$, 
$$\|\lambda_s g \lambda_s^*- \lambda_s h \lambda_s^*\|_{C^*_r(G)}=\|\lambda_s (g-h) \lambda_s^* \|_{C^*_r(G)}= \|g-h\|_{C^*_r(G)},$$
we may assume that $x=\lambda(f)$, $f\in C_c(G)$. Now, by our assumption, there are $\epsilon>0$ and $\phi\in B_r(G)=C^*_r(G)^*$ such that for every $s\in G$, 
$$\left|\int_G \phi(t)f(s^{-1}ts) dt \right|\geq \epsilon.  $$
On the other hand, by \cite[Theorem, p.211]{cowling},  $G$ has Kunze-Stien property  so that there is $p>2$ such that $\phi \in L^p(G)$ (in fact, this holds for every $p>2$).  Let $\psi\in C_c(G)$ with 
$$\|\phi-\psi\|_{L^p(G)}< \frac{\epsilon}{2\|f\|_{L^q(G)}},$$
where $q$ is the conjugate of $p$ satisfying $\frac{1}{p}+\frac{1}{q}=1$.
Then, by combining the above relations, letting 
$$(\gamma_G(s)g)(t)=g(s^{-1}ts)  \ \ \ (s,t\in G, g\in L^p(G)),$$
and noting that $G$ is unimodular and hence $\gamma_G(s)$ is an isometry, we have 
\begin{align*}
    \epsilon &\leq \left|\int_G \phi(t)f(s^{-1}ts) dt \right| \\
     &\leq \left|\int_G (\phi(t)-\psi(t)) f(s^{-1}ts) dt \right|+\left|\int_G \psi(t)f(s^{-1}ts) dt \right| \\
     &\leq \|\phi-\psi\|_{L^p(G)} \|\gamma_G(s)f\|_{L^q(G)}+\left|\int_G \psi(t)f(s^{-1}ts) dt \right| \\
     &= \|\phi-\psi\|_{L^p(G)} \|f\|_{L^q(G)}+\left|\int_G \psi(t)f(s^{-1}ts) dt \right| \\
     &\leq \frac{\epsilon}{2}+|( \gamma_G(s)f, \bar\psi )|.
\end{align*}
Hence we have 
$$|( \gamma_G(s)f, \bar\psi )| \geq \frac{\epsilon}{2} \ \ \ (s\in G).$$
Since $f,\psi \in C_c(G)\subset  L^2(G)$, we must then have that $\gamma_G$ is not weakly mixing \cite[Corollary 1.6]{ber_ros}.  Hence it must have a finite-dimensional subrepresentation \cite[Theorem 1.9]{ber_ros}. 
In particular, it is an amenable representation in the sense of Bekka, \cite[Theorem 1.3]{bekka}. But that is not possible because, by \cite[Theorem 2.4]{bekka}, the conjugate representation is amenable if and only if $G$
is an inner-amenable group. However, $G$ is connected so that it is inner-amenable if and only if it is amenable (see \cite[Theorem 1]{losert_rindler}), and hence compact by \cite[Theorem 8.7.6]{re}, which is not the case here. 
\end{proof}

Next theorem gives another natural class of weight inverses  for which the inclusion $M_{\rm cb}A(G,\omega)\subset M_{\rm cb}A(G)$ holds. 
\begin{theorem}
    Let $G$ be a locally compact group, and let $H$ be a normal amenable subgroup. Suppose that $\omega_H$ is a weight inverse on $H$ and 
    $\omega=\gamma(\omega_H)$ is the corresponding weight inverse on $G$. Then $M_{\rm cb}A(G,\omega)$ is a subspace of $M_{\rm cb}A(G)$.
\end{theorem}

\begin{proof}
Let $\phi\in M_{\rm cb}A(G,\omega)$. Since $H$ is normal, $G$ acts on $H$ by conjugation:
\begin{align*}
s\cdot h:=shs^{-1} \ \ \ (s\in G,h\in H).     
\end{align*}
This induces an action on $VN(H)$ having both $C^*_r(H)$ and $M(C^*_r(H))$ as $G$-invariant C$^*$-subalgebras. In particular,
$$\lambda_s^*\gamma(C_r^*(H))\lambda_s\subset\gamma(C_r^*(H)).$$
Let $x\in C_r^*(H)$ and write $x_s$ for the preimage $\gamma^{-1}(\lambda_s\gamma(x))\lambda_s^*)\in C_r^*(H)$. 
As $H$ is amenable, the functional $1\in L^\infty(H)$, $1:\lambda(g)\mapsto\int_Hg(h)dh$, $g\in L^1(H)$, extends uniquely to a multiplicative linear functional on $M(C_r^*(H))$ (see e.g. the proof of \cite[Lemma 4.2]{ors} for details).  Let $\chi:\gamma (C_r^*(H))\to \mathbb C$ be given by 
$$\chi(\gamma(a))=\langle a,1\rangle, a\in C_r^*(H).$$ Observe that for every $g\in L^1(H)$, 
$$\gamma(\lambda^H(g)_s)=\lambda_s^*\gamma(\lambda^H(g))\lambda_s=\int_Hg(t)\lambda_{s^{-1}\cdot t}^Gdt.$$
Therefore
$$\chi (\gamma(\lambda^H(g)_s))=\int_Hg(t)\langle\lambda_{s^{-1}\cdot t}^G , 1 \rangle dt=\int_H g(t)dt=\langle \lambda^H(g),1\rangle$$ and hence 
$$\chi (\gamma(a_s))=\langle a,1\rangle, \ \ a\in C_r^*(H).$$ 
Let $\phi\in M_{\rm cb}A(G,\omega)$. By extending $\chi$ to a positive linear functional on $B(L^2(G))$,
it follows from the proof of Theorem \ref{mcb} that $(t,s)\mapsto \phi(ts^{-1})\chi(\gamma((\omega_Hf)_s))$ is a Schur multiplier.
As $\chi(\gamma((\omega_Hf)_s))=\langle\omega_Hf,1\rangle\ne 0$ for some $f\in C_r^*(H)$ by the density condition on $\omega_H$, we obtain that
the mapping $(s,t)\mapsto \phi(ts^{-1})$ is a Schur multiplier. This completes the proof.

\end{proof}





It is known, for a locally compact group $G$, $M_{\rm cb}(A(G))$  
is a dual Banach algebra with the predual $Q(G)$ obtained by completion of $L^1(G)$ in the norm
$$\|f\|_{Q(G)}=\sup\left\{\left|\int_G f(s)u(s)ds\right|: u\in M_{\rm cb}A(G), \|u\|_{\rm cb}\leq 1\right\},$$
see \cite[Proposition 1.10]{canniere-haagerup}.
In the following, we show that this is also the case for  $M_{\rm cb}A(G, \omega )$.
To prove this, we need the following lemma.

\begin{lemma}\label{l:A(G,w)-dense-product}
Let $G$ be a locally compact group and $\omega $ be a weight inverse on it. Then the  linear space spanned by $\{(\omega u)(\omega v): u,v \in A(G)\}$ is dense in $A(G,\omega )$.
\end{lemma}

\begin{proof}
Let $\cA$ be  the closure of the linear subspace constructed by $\{(\omega u)(\omega v): u,v \in A(G)\}$ in $A(G,\omega )$.
If the inclusion $\cA \subseteq A(G,\omega )$ is  proper, there is non-zero  $x \in VN(G)$   such that  the restriction of $x$ to $\cA$ is $0$. Therefore, as shown in the proof of Theorem \ref{t:main},
\begin{eqnarray*}
0 = \langle (\omega u)(\omega v) , x\rangle_\omega   &=& \langle \omega  \Gamma_*(\Omega (u \otimes v)) , x\rangle_\omega  \\
&=& \langle  \Omega (u \otimes v) ,\Gamma( x)  \rangle\\
&=& \langle  u \otimes v , \Gamma( x)  \Omega \rangle.
\end{eqnarray*} 
Since $A(G) \odot A(G)$ is dense in $A(G \times G)$, we get that $A(G \times G)$ is a subset of the kernel of $\Gamma( x)  \Omega $.
Let $\xi, \eta  \in L^2(G \times G)$ be chosen  arbitrarily.   Therefore, for $u\in A(G\times G)$ given by
$$u(t,s)=((\lambda_t\otimes \lambda_s)\xi, \eta) \ \ \  \ (t,s\in G),$$
we get
\[
0 = \langle u, \Gamma( x)  \Omega  \rangle = (\Gamma( x)  \Omega   \xi,  \eta).
\]
But since  by Lemma~\ref{l:ker(w-1)}, $\ker(\Omega )=\{0\}$, $\Omega $ has a dense range.  Therefore, the last equality implies that $\Gamma(x) =0$ which is a contradiction. 
\end{proof}

Let $A$ be a completely contractive Banach algebra, $X$ is right $A$-module  and $Y$ is a left $A$-module. The $A$-module tensor product of $X$ and $Y$ is the quotient space $X\hat{\otimes}_AY := X\hat{\otimes} Y/N$, where
\[
N =\langle x\cdot a\otimes y - x\otimes a\cdot  y\; |\; x\in X, y\in Y, a\in A\rangle,
\]
and  $\langle \cdot \rangle$  denotes the closed linear span. It follows that
\[
CB_A(X,Y^*) \cong N^\perp  \cong (X \hat{\otimes}_AY)^*,
\]
where $CB_A(X, Y^*)$ denotes the space of completely bounded right $A$-module maps $\Phi : X \rightarrow Y^*$ (see \cite[Corollary 3.5.10]{blecher-lemerdy}).

In \cite[Remark~5.9]{jay}, J. Crann showed that, for a locally compact group $G$, $M_{\rm cb}A(G)$ is completely isometric to the dual space of $A(G)\hat{\otimes}_{A(G)} C^*_r(G)$.
In the theorem below, we  extend this result to Beurling-Fourier algebras and show that $M_{\rm cb}A(G,\omega)$ is a dual operator space. 
It is known that $M_{\rm cb}A(G)$ is a dual Banach algebra. We believe that this should hold for  $M_{\rm cb}A(G,\omega)$ as well.
By Theorem \ref{t:G-amenable-MA(G,w)} and Lemma \ref{l:B(G,omega)-dual-Banach-algebra}, this is true for amenable groups.

\begin{theorem}\label{t:McbA(G,w)-dual-Balgebra}
Let $G$ be a locally compact with a  weight inverse $\omega $. Then $M_{\rm cb}A(G,\omega )$ is  completely isometric to the dual space of  $A(G,\omega )\hat{\otimes}_{A(G,\omega )} C^*_r(G)$.
\end{theorem}

\begin{proof}
Consider $A(G,\omega ) \hat{\otimes} C^*_r(G)$. Then it is know that $(A(G,\omega ) \hat{\otimes} C^*_r(G))^*$ is isometrically isomorphic to $CB(A(G,\omega ), B_r(G, \omega ))$ by the equation \cite[7.1.15]{er-book}.  Now let $X$ be the closed linear subspace of $A(G,\omega ) \hat{\otimes} C^*_r(G)$ given  by
\[
X:=\left\langle (\omega  u) (\omega v)\otimes x) -  (\omega u \otimes (\omega v) \cdot_\omega  x): u,v \in A(G,\omega ), x \in C^*_r(G)\right\rangle.
\]
Then $X^\perp \subseteq CB(A(G,\omega ), B_r(G,\omega ))$  is a weak$^*$-closed subspace. Also, for each $T$  in $CB(A(G,\omega ), B_r(G,\omega ))$ we have that  $T\in X^\perp$  if and only if $T((\omega  u) (\omega v))=T(\omega  u) (\omega v)$.  But, by Lemma~\ref{l:A(G,w)-dense-product}, we know that products are dense in $A(G,\omega )$ and, by Theorem~\ref{t:main}, $A(G,\omega )$ is an ideal in $B_r(G,\omega )$. Therefore, $T(A(G,\omega )) \subseteq A(G,\omega )$. Consequently, there is some $\phi \in M_{\rm cb}A(G,\omega )$ so that $\bm_\phi=T$. 
We have thus shown that $M_{\rm cb}A(G,\omega )$ is completely isometric to $X^\perp$, and hence $M_{\rm cb}A(G,\omega )$ is a dual Banach space where the weak$^*$ topology on $M_{\rm cb}A(G,\omega )$ is derived through embedding $M_{\rm cb}A(G,\omega ) \rightarrow  (A(G,\omega )\hat{\otimes} C^*_r(G))^*$ given by
\[
\langle \bm_\phi, (\omega u)\otimes x) \rangle = \langle \bm_\phi(\omega  u), x\rangle_\omega .
\]


\end{proof}

\section{Completely bounded multipliers of $A(G,\omega )$ for   amenable groups}\label{sec_amen}

In the final section, we consider Beurling-Fourier algebras on locally compact amenable groups. We start with the following, which characterize the amenability of the group in terms of (bounded approximate) identities.

\begin{proposition}\label{p:bai-identity}
Let $G$ be a locally compact group with a weight inverse $\omega $. Then the following conditions are equivalent.
\begin{itemize}
\item[(i)]{$G$ is amenable.}
\item[(ii)]{$A(G, \omega )$ has a bounded approximate identity.}
\item[(iii)]{$B_r(G, \omega )$ includes the constant function $1$.}
\end{itemize}
\end{proposition}

\begin{proof}
(i)$\Leftrightarrow$(ii) is the main result of \cite{ors}.

(ii)$\Rightarrow$(iii). Let $(v_\alpha)$ be in $A(G)$ so that $(\omega  v_\alpha)_\alpha$ is a bounded approximate identity of $A(G, \omega )$.
Since the unit ball of $B_r(G, \omega )$ is weakly$^*$ compact (as a dual Banach space) and contains the net $(\omega v_\alpha)$, there is a weak$^*$ cluster point $\omega v_0 $ for the net in $B_r(G, \omega )$.  We claim that $\omega v_0$ is the constant function $1$.  To do so, first note that for each $u \in A(G)$ and $x \in C^*_r(G)$, we have
\begin{eqnarray}\label{eq:weak*-limit}
\langle (\omega u)(\omega v_0), x\rangle_\omega  &=& \langle (\omega v_0), x \cdot_\omega  (\omega u)\rangle_\omega \nonumber\\
 &=& \lim_\alpha \langle (\omega v_\alpha), x \cdot_\omega  (\omega u)\rangle_\omega \\
  &=& \lim_\alpha \langle (\omega u)(\omega v_\alpha), x \rangle_\omega \nonumber\\
  &=&  \langle \omega u, x \rangle_\omega \nonumber
\end{eqnarray}
where (\ref{eq:weak*-limit}) is justified due to the fact that $B_r(G,\omega )$ is a dual Banach algebra, by Lemma~\ref{l:B(G,omega)-dual-Banach-algebra}, and $(v_\alpha)_\alpha$ may have been replaced by one of its subnets.
Hence 
$ (\omega u)(\omega v_0) =  \omega u$ for every $ \omega u \in A(G, \omega )$. Therefore,  we have that $\omega v_0\in B_r(G,\omega)$ is the constant function $1$ 

(iii)$\Rightarrow$(i) is immediate if we recall that $B_r(G, \omega ) \subseteq B_r(G)$ and $1 \in B_r(G)$ if and only if $G$ is amenable.
\end{proof}

Recall that for each $u \in A(G)$ and $x \in VN(G)$ we have
\begin{eqnarray}\label{eq:tbm}
\langle \bm_\phi(\omega  u), x\rangle_\omega  &=& \langle \omega  u, \bm^*_\phi( x)\rangle_\omega \\
&=& \langle  u, \bm^*_\phi( x)\rangle.\nonumber
\end{eqnarray}
Now let us define $\tbm_\phi(u)$ given by
\[
\langle \tbm_\phi(u), x\rangle := \langle  u, \bm^*_\phi( x)\rangle \quad \quad \quad (x\in VN(G)).
\]
Note that based on \eqref{eq:tbm}, $\tbm_\phi(u)$ is a normal map on $VN(G)$. Equivalently, $\tbm_\phi(u) \in A(G)$ for every $u \in A(G)$.   One can readily show that for each $u \in A(G)$,

\begin{equation}\label{eq:omega-tbm-commuting}
\bm_\phi(\omega  u) = \omega  \tbm_\phi(u).
\end{equation}

\begin{proposition}\label{p:m(C*)-in-C*}
Let $G$ be an amenable locally compact group, and let $\omega $ be a weight inverse on it. Then for each $\phi \in MA(G, \omega )$, $\bm^*_\phi$ maps $C^*_r(G)$ into $C^*_r(G)$.  Furthermore, $\phi$ is also a multiplier of $B_r(G, \omega )$. 
\end{proposition}

\begin{proof}
Let $x \in C^*_r(G)$. Since $G$ is amenable, by Proposition~\ref{p:bai-identity}, $A(G, \omega )$ has a bounded approximate identity. Also $C^*_r(G)$ is an $A(G, \omega )$-bimodule, by Lemma~\ref{l:B(G,omega)-dual-Banach-algebra}. Therefore, by Cohen's factorization theorem, there is a $v \in A(G)$ and $x'\in C^*_r(G)$ so that $ x = \omega v \cdot_\omega  x'$. Now for each $u \in A(G)$ we have
\begin{eqnarray*}
\langle \omega u,\bm_\phi^*(x)\rangle_\omega  &=& \langle \bm_\phi(\omega )u, x\rangle_\omega \\
&=& \langle \bm_\phi(\omega u), \omega  v \cdot_\omega  x'\rangle_\omega \\
&=& \langle \bm_\phi(\omega u)\omega  v  ,   x'\rangle_\omega \\
&=& \langle\omega u \bm_\phi(\omega  v)  ,   x'\rangle_\omega \\
&=& \langle(\omega u) (\omega  \tbm_\phi(v))  ,   x'\rangle_\omega \\
&=& \langle \omega  \Gamma_*(\Omega ( u \otimes  \tbm_\phi(v))  ,   x'\rangle_\omega \\
&=& \langle  \Gamma_*(\Omega ( u \otimes  \tbm_\phi(v))  ,   x'\rangle\\
&=& \langle    u \otimes  \tbm_\phi(v)  ,   \Gamma(x') \Omega \rangle.
\end{eqnarray*}
Now note that $ \tbm_\phi(v) \in A(G)$. Since $A(G)$ is an $C^*_r(G)$-bimodule and $C^*_r(G)$ has a bounded approximate identity, we have $v' \in A(G)$ and $y \in C^*_r(G)$ so that $ \tbm_\phi(v) = v ' \cdot y$. Hence, we have
\begin{eqnarray*}
\langle \omega u,\bm_\phi^*(x)\rangle_\omega  &=& \langle    u \otimes  (v' \cdot y)  ,   \Gamma(x') \Omega \rangle\\
&=& \langle   ( u \otimes  v')\cdot(1 \otimes y)  ,   \Gamma(x') \Omega \rangle\\
&=& \langle    u \otimes  v'  ,   (1 \otimes y)\Gamma(x') \Omega \rangle\\
&=& \langle    u  ,  \id  \otimes  v' \left( (1 \otimes y)\Gamma(x') \Omega \right) \rangle\\
&=& \langle   \omega  u  ,  \id  \otimes  v' \left( (1 \otimes y)\Gamma(x') \Omega \right) \rangle_\omega \\
\end{eqnarray*}
But similar to the former proofs, $(1 \otimes y)\Gamma(x')\Omega \in C^*_r(G) \otimes C^*_r(G)$, by Lemma \ref{mult}. So by  slicing by $v' \in A(G) \subseteq B(G)=C^*_r(G)^*$, we have that $ \id  \otimes  v' \left( (1 \otimes y)\Gamma(x') \Omega \right) $ belongs to $C^*_r(G)$. Since $u$ was chosen arbitrarily, this proves that $\bm_\phi^*(x) \in C^*_r(G)$.

Next, we prove that $\left(\bm_\phi^{*}|_{C^*_r(G)}\right)^*$ is a multiplier of $B_r(G,\omega )$. To do so, first note that for each $x \in C^*_r(G)$ and $u, v\in A(G)$, we have
\begin{eqnarray*}
\langle u\otimes v, \Gamma(\bm_\phi^*(x))\Omega  \rangle &=& \langle  \Omega (u\otimes v), \Gamma(\bm_\phi^*(x)) \rangle\\
&=&  \langle \Gamma_*( \Omega (u\otimes v)), \bm_\phi^*(x) \rangle\\
&=&  \langle \omega \Gamma_*( \Omega (u\otimes v)), \bm_\phi^*(x) \rangle_\omega \\
&=&   \langle \bm_\phi( (\omega u)(\omega  v)), x \rangle_\omega \\
&=&  \langle \bm_\phi(\omega u)\;  \omega  v, x \rangle_\omega \\
&=&  \langle \omega \Gamma_*(\Omega (\tbm_\phi(u) \otimes v), x \rangle_\omega \\
&=&  \langle  \Gamma_*(\Omega (\tbm_\phi(u) \otimes v), x \rangle\\
&=&  \langle  \tbm_\phi(u) \otimes v, \Gamma(x) \Omega  \rangle\\
&=&  \langle  u \otimes v,\bm_\phi^*\otimes \id( \Gamma(x) \Omega  )\rangle.
\end{eqnarray*}
Since the algebraic tensor product $A(G) \odot A(G)$ is dense in $A(G\times G)$, we have 
$$\Gamma(\bm_\phi^*(x))\Omega  = \bm_\phi^*\otimes \id( \Gamma(x) \Omega  ).$$

Finally, consider $\bn_\phi:=\left(\bm_\phi^{*}|_{C^*_r(G)}\right)^*$. Similarly to the definition of  $\tbm_\phi$, one can prove that there is a bounded  mapping $\tbn_\phi: B_r(G) \rightarrow B_r(G)$ so that
\[
\bn_\phi(\omega  u) = \omega  \tbn_\phi(u)
\]
for every $u \in B_r(G)$. Now let $u,v \in B_r(G)$ and $x \in C^*_r(G)$. Then 
\begin{eqnarray*}
\langle \bn_\phi((\omega u) (\omega  v)), x\rangle_\omega  &=& \langle (\omega u) (\omega  v)), \bm_\phi^*(x)\rangle_\omega \\
&=& \langle\omega  \Gamma^*(\Omega  (u \otimes v)), \bm_\phi^*(x)\rangle_\omega \\
&=& \langle\Gamma^*(\Omega  (u \otimes v)), \bm_\phi^*(x)\rangle\\
&=& \langle u \otimes v, \Gamma(\bm_\phi^*(x))\Omega \rangle\\
&=& \langle u \otimes v,\bm_\phi^*\otimes \id( \Gamma(x) \Omega  )\rangle.
\end{eqnarray*}
Here note that $\Gamma(x) \Omega $ is an element in $M(C^*_r(G) \otimes C^*_r(G))$. So using universality of the multiplier algebra of $C^*$-algebras, one may consider $M(C^*_r(G) \otimes C^*_r(G))$ as a subspace of $(C^*_r(G)\otimes C^*_r(G))^{**}$ and therefore, $(\bm_\phi^*\otimes \id)^*|_{B_r(G\times G)}$ would be nothing but $\tilde\bn_\phi\otimes \id$. So we have
\begin{eqnarray*}
\langle \bn_\phi((\omega u) (\omega  v)), x\rangle_\omega  &=&  \langle \tbn_\phi(u) \otimes v,\Gamma(x) \Omega  \rangle\\
&=&  \langle \Gamma^*(\Omega (\tbn_\phi(u) \otimes v)),x \rangle\\
&=&  \langle \omega  \Gamma^*(\Omega (\tbn_\phi(u) \otimes v)),x \rangle_\omega \\
&=&  \langle (\omega \tbn_\phi(u)) \omega v,x \rangle_\omega \\
&=&  \langle \bn_\phi(\omega u) \omega v,x \rangle_\omega .
\end{eqnarray*}
This proves the desired result.
\end{proof}

\begin{theorem}\label{t:G-amenable-MA(G,w)}
Let $G$ be a locally compact group and $\omega $ be a weight inverse on it. Then $MA(G,\omega ) = M_{\rm cb}A(G,\omega )=B_r(G, \omega )$ if and only if $G$ is amenable.
\end{theorem}

\begin{proof}
Suppose that  $G$ is amenable then, by Proposition~\ref{p:bai-identity}, $A(G,\omega )$ is a completely contractive Banach algebra which has a bounded approximate identity.
But  \cite[Theorem~6.2]{daws-mul}  implies that for such a Banach algebra, every multiplier is completely bounded, i.e. $MA(G,\omega )=M_{\rm cb}A(G, \omega )$.

Also, since $A(G, \omega )$ is an ideal in $B_r(G, \omega )$, $B_r(G,\omega )$ can be considered as a subset of $MA(G,\omega )$. On the other hand, by  Proposition~\ref{p:m(C*)-in-C*}, every element in $MA(G, \omega )$ can be considered as a multiplier of $B_r(G, \omega )$. But since $G$ is amenable, by Proposition~\ref{p:bai-identity}, $1 \in B_r(G, \omega )$. This implies that $MA(G, \omega )$ is a subset of $B_r(G,\omega )$. Ergo, $MA(G,\omega )=B_r(G,\omega )$.

Conversely if $MA(G, \omega )=B_r(G, \omega )$, then the constant function $1$ belongs to  $B_r(G,\omega )$.  Then by Proposition~\ref{p:bai-identity}, $G$ has to be amenable.
\end{proof}

\bigskip

{\bf Acknowledgements.} E.S. and L.T. would like to thank the Wenner-Grenn Foundation (the project GFOh2024-0025) which supported the sabbatical visit of Ebrahim Samei to Chalmers University of Technology, 2025-2026. Research of L.T. was supported by the Swedish Research Council project grant 2023-04555. E.S. was partially supported by NSERC Discovery Grant RGPIN-2025-04833. O.G., E.S. and L.T. thank the Swedish Research Council for their support under grant no.\ 2021-06594 while the authors were in residence at Institut Mittag-Leffler in Djursholm, Sweden. 

We would like thank Nico Spronk and Matthew Wiersma  for discussions and providing the references \cite{bekka, fsw, KenRuam}. 

\footnotesize

\bibliographystyle{plain}
\bibliography{Bibliography}

\noindent Mahmood Alaghmandan:\\
 Address: Model Validation and Risk Management, Farm Credit Canada, 100 Queen St no.1460, Ottawa, ON K1P 1J9 Canada\\
e-mail: mahmood.alaghmandan@fcc.ca

\medskip 

\noindent Olof Giselsson:\\
\noindent Address: Department of Mathematical Sciences, Chalmers University of Technology and University of Gothenburg, Chalmers Tv{\"a}rgata 3, Gothenburg, SE-412 58,  Sweden \\
e-mail: olofgiselsson@proton.se

\medskip

\noindent Ebrahim Samei:\\
\noindent Address: Department of Mathematics and Statistics, University of Saskatchewan, Saskatoon, Saskatchewan, S7N 5E6, Canada\\
e-mail: ebrahim.samei@usask.ca

\medskip
\noindent Lyudmila Turowska:\\
\noindent Address: Department of Mathematical Sciences, Chalmers University of Technology and University of Gothenburg, Chalmers Tv{\"a}rgata 3, Gothenburg, SE-412 58,  Sweden \\
e-mail: turowska@chalmers.se
\end{document}